\documentclass[11pt,reqno,a4paper]{amsart}
\usepackage{graphicx}
\usepackage{enumitem}

\usepackage{amsmath}
\usepackage{amssymb}
\usepackage{amsthm}

\usepackage{xifthen}

\usepackage{aliascnt}
\usepackage[bookmarksopen,bookmarksdepth=2,backref=none,hidelinks]{hyperref}
\usepackage{cleveref}

\theoremstyle{definition}
\usepackage{esint}

\theoremstyle{plain}

\newcommand{\floor}[1]{\lfloor #1 \rfloor}

\DeclareMathOperator*{\argmin}{arg\,min}
\DeclareMathOperator*{\cof}{cof}

\renewcommand{\epsilon}{\varepsilon}
\newcommand{\norm}[1]{\left\|#1\right\|}
\newcommand{\abs}[1]{\left|#1\right|}
\newcommand{\dx}[1][x]{\, \mathrm{d}#1}
\newcommand{\dt}{\dx[t]}
\newcommand{\R}{\mathbb{R}}
\newcommand{\N}{\mathbb{N}}
\newcommand{\weakto}{\rightharpoonup}

\renewcommand{\div}{\operatorname{div}}

\newtheorem{theorem}{Theorem}
\newtheorem{lemma}[theorem]{Lemma}
\theoremstyle{definition}
\newtheorem{definition}[theorem]{Definition}
\newtheorem{assumptions}[theorem]{Assumptions}
\newtheorem*{remark}{Remark}

\numberwithin{theorem}{section}

\usepackage{xcolor}

\numberwithin{equation}{section}

\allowdisplaybreaks


\title[In time convergence of hyperbolic approximations]{Stability and convergence of in time approximations of hyperbolic elastodynamics via stepwise minimization}
\author{Anton\'in \v Ce\v s\'ik}
\address{Faculty of Mathematics and Physics, Charles University, Sokolovsk\'{a} 83, 18675, Prague, Czech Republic} \email{cesik@karlin.mff.cuni.cz}

\author{Sebastian Schwarzacher}
\address{Department of Mathematics, Uppsala University, Box 480
    751 06 Uppsala \&  Faculty of Mathematics and Physics, Charles University,
Sokolovsk\'{a} 83, 18675, Prague, Czech Republic }
\email{schwarz@karlin.mff.cuni.cz}

\begin{document}
 \begin{abstract}
We study step-wise time approximations of non-linear hyperbolic initial value problems. The technique used here is a generalization of the minimizing movements method, using two time-scales: one for velocity, the other (potentially much larger) for acceleration.
The main applications are from elastodynamics namely so-called generalized solids, undergoing large deformations. The evolution follows an underlying variational structure exploited by step-wise minimisation.
 We show for a large family of (elastic) energies that the introduced scheme is stable; allowing for non-linearities of highest order. If the highest order can assumed to be linear, we show that the limit solutions are regular and that the minimizing movements scheme converges with optimal linear rate. Thus this work extends numerical time-step minimization methods to the realm of hyperbolic problems.
\end{abstract}

 \maketitle

 \tableofcontents


\section{Introduction}

%

We study step-wise time approximations of hyperbolic non-linear initial value problems. For this we consider $Q\subset \R^n$ a bounded Lipschitz domain and a time interval $[0,T]$. 
The partial differential equations considered here are of the form
\begin{align}
\label{eq:intro}
\begin{aligned}
&\partial_{tt} \eta(t,x) + DE(\eta(t,x))=f(t,x)\text{ for }(t,x)\in [0,T]\times \Omega
\\
&\eta(0,x)=\eta_0(x),\quad \partial_t\eta(0,x)=\eta_*(x)\text{ for }x\in Q
\end{aligned}
\end{align}
where $E$ is some energy functional, $DE$ its (Fr\'echet) derivative, $\eta_0$, $\eta_*$ given initial conditions and $f$ a given right hand side.
We supplement the problem with prescribed boundary values. Our motivational example are the dynamics of  \emph{largely deforming} elastic solids. Therefore one critical challenge of this paper is to allow for elastic energies that include negative powers of the Jacobian (see \eqref{kelvinVoigt}). This means the energies can be non-convex and be defined over a {\em non-convex state space} (as for example the space of vectorfields with positive Jacobians a.e.).

In the case of {\em a convex state space}, implicit Euler or variational schemes have been studied exhaustively, see for instance \cite{Friesecke1997,Carstensen2004,Demoulini2001, Prohl2008,Haehnle2010}.  
Further in situations where additionally the energy is assumed to be convex more approaches are applicable, see for instance the classical works~\cite{Kacur1986,Pultar1984}. Note that none of the above literature is applicable for largely deforming solid evolutions. In particular, none treats non-convex state spaces. Actually, it seems that the case of a hyperbolic PDE describing largely deforming solid evolutions has not been treated before. There are so far only numeric results for the quasi-static and visco-elastic case~\cite{barrett2010numerical,RouTso21}.

We follow the scheme developed in \cite{BenesovaKampschulteSchwarzacher} where via step-wise minimization a second order in time evolution was approximated. The heart of the method was to use two different time scales -- the {\em velocity scale $\tau$} and the (potentially much larger) {\em acceleration scale $h$}. Accordingly, $\partial_{tt}\eta$ is approximated as

\begin{equation*}
\partial_{tt}\eta(t)\approx \frac{\frac{\eta(t)-\eta(t-\tau)}{\tau}-\frac{\eta(t-h)-\eta(t-h-\tau)}{\tau}}{h}
\end{equation*}
In \cite{BenesovaKampschulteSchwarzacher} the fact was exploited that for a fixed acceleration scale a gradient flow structure can be used naturally on the length scale of $h$. This opened the door to use variational strategies for hyperbolic evolutions.

{\bf In this paper we wish to investigate the potential of the method for numerical discrete-in-time approximations.}

We consider this question worthy since step-wise minimisation is a rather well established approximation strategy for gradient flows ever since the seminal works of DeGiorgi~\cite{DeGiorgi}.
It was used widely both for analysis and numerics, see ~\cite{Kru98,mayer2000numerical,barrett2010numerical,
BarKru11,gigli2013entropic,
laux2016convergence,
RinSchSul17,MieRou20,RouTso21} and the references therein for some examples of applications, but there are many more.

In this work we analyse the potential of step-minimisation for numerics of hyperbolic problems, by providing respective stability and convergence results. 

The main motivating application are hyperbolic evolutions for {\em elastic solids that may deform largely}~\cite[Chapter 3]{BenesovaKampschulteSchwarzacher}. 
A typical example for the elastic energy (see \cite{Ball1976,Dacorogna2007}) is
\begin{equation}
\label{kelvinVoigt}
E_1(\eta) := \begin{cases}
 \displaystyle\int_Q \frac{1}{8}\big(\mathcal{C} (\nabla \eta^T \nabla \eta- I) \big) \cdot \big( \nabla \eta^T \nabla \eta- I \big) + \frac{1}{(\det \nabla \eta)^a}dx, &\text{for $\eta$, $\det \nabla \eta > 0$ a.e.\ in $Q$ }
 \\
+ \infty & \text{otherwise.} 
\end{cases}
\end{equation}  Here $\mathcal{C}$ is a positive definite tensor of elastic constants and $a$ a given exponent. In $E_1$ the first term corresponds to the Saint Venant-Kirchhoff energy, the second models the resistance of the solids to infinite compression.

{\em It is the non-convexity of the state space (not of the functional) that precludes using other approaches such as fixed point methods}; but a minimizer can still be found~\cite{Ball1976,Dacorogna2007}. Observe that this non-convexity stems from the physical requirement that compressing a bulk solid into zero volume requires infinite energy. Hence, this example is a strong motivation for the {\em utility and potential necessity of variational schemes}.

Although largely deforming elastic solids are our main motivation, the scheme has the potential to be used in other possible more complicated scenarios. Applications of elastodynamics with good perspective of application of the methods here include fluid-structure interaction scenarios~\cite{BenesovaKampschulteSchwarzacher}, \cite{Benesova2023}, \cite{Kampschulte2023}, (elasto-)plastic motions~\cite{MieRou06,MieRou16,Mielke2017,RinSchSul17} or including damage~\cite{MieRou06} or temperature~\cite{MieRou20,BadFriKru22}. For that reason we additionally provide abstract assumptions for stability and convergence. 

The setting also includes the case of ordinary differential equations in $\R^n$. In this setting the results are illustrate with some numerical experiments showing the sharpness of the estimates in \Cref{sec:numerics}.

%
 The current state of the art in the analysis for {\em evolutionary large deformation solids} conventionally assumes the elastic energy also to depend on higher order derivatives.
 Such solids are known as {\em hyperelastic solids}~\cite{Nguyen2000,Healey2009,Doghri2000,Kruzik2019} and conventionally add to $E_1$ a second functional $
 E_2\colon W^{k,p}(Q)\to [0,\infty]
 $
 which is typically of the form
\begin{align}
\label{e2:nonlin}
E_2(\eta)&=\int_Q(1+\abs{\nabla^2\eta})^{q-2}\abs{\nabla^2\eta}^2\, dx\text{ or }
\\
\label{e2:lin}
E_2(\eta)&=\int_Q\abs{\nabla^k\eta}^{2}\, dx.
\end{align}
For the latter one $DE_2$ is a linear operator. 

{\em In this paper we provide stability and convergence results for the variationally constructed two-scale time-discrete
solutions.}

All here demonstrated results are valid for arbitrary choices of $\tau$ and $h$, provided they are sufficiently small in relation to the non-convexity. This might seem surprising with regard to the fact that in the analytic results using the hyperbolic variational scheme the independent limit passage is essential. This is one of the reasons we believe the here demonstrated estimates to be relevant for pure analytic applications independently from its value for numerics, see also~\Cref{rem:existence}.

We include here the main results for the model case and $\tau=h$. The general results for arbitrary $\tau$ and $h$ and for general assumptions on $DE(\eta)$ can be found in \Cref{thm:stability-solid}, \Cref{thm:convergence-rate-solid} and \Cref{thm:higher-space-regularity}.

Our main stability theorem deals with the approximation \eqref{def:eta-minimizing-movements} which includes a dissipation term (vanishing in the limit) which stabilizes the scheme so that there is no increase in energy with time.

\begin{theorem}[Stability for the scheme with artificial viscosity]
 Let $E(\eta)=E_1(\eta)+E_2(\eta)$, where $E_1$ is given by \eqref{kelvinVoigt} and $E_2$ by \eqref{e2:nonlin} or \eqref{e2:lin} satisfying \eqref{eqn:alpha0}. Let $\eta_k$ be the variational approximation obtained by step-wise minimization 
 \begin{equation}\label{def:eta-minimizing-movements-hequaltau}
\eta_k = \argmin_{\eta\in\mathcal E}\frac{\tau^2}{2}\norm{\frac{\frac{\eta-\eta_{k-1}}{\tau}-\frac{\eta_{k-1}-\eta_{k-2}}{\tau}}{\tau}
}_{L^2}^2+E(\eta) +\frac{c\tau^2}{2}\norm{\nabla\frac{\eta_k-\eta_{k-1}}{\tau}}_{L^2}^2 -\langle f_k , \eta\rangle_{L^2},
\end{equation}
where $\mathcal E$ is defined in \eqref{eq:E} and $c>0$.
Then there exists a $\tau_0>0$ and a $c_0>0$ depending on the assumptions on $E$, the initial data and the right hand side, such that for all $0<\tau\leq\tau_0$ and for $c\geq c_0$, the stability estimate 
\begin{equation*}
E(\eta_k) +\frac12\norm{\frac{\eta_k-\eta_{k-1}}{\tau}}_{L^2}^2  \leq  \left(E(\eta_0)+\frac12 \norm{\eta_*}_{L^2}^2+ \frac12\norm{f}_{L^2((0,T);L^2)}^2\right),
\end{equation*}
is satisfied
whenever the right hand side is finite and $(\eta_k)_{k=1,\dots,\floor{T/\tau}}$ does not reach a collision. 
\end{theorem}

Next we show that (given uniform convexity of the leading term) one obtains an energy estimate also for the direct approach, meaning without the stabilization term. In this case energy can increase in time even in the case $f=0$, however note that the loss is controlled by $(1+4C\tau h \ell)\leq (1+4C\tau T)\to 1$ with $\tau\to 0$. Though this stability estimate is weaker than the previous one, we include it as it is of possible interest due to the easier implementation of this scheme.

\begin{theorem}[Stability of the direct approach]
Let $E$ be as above and moreover $E_2$ be uniformly convex. Let $\eta_k$ be the variational approximation obtained by step-wise minimization  \begin{equation}\label{def:eta-minimizing-movements-nodiss-hequaltau}
\eta_k = \argmin_{\eta\in\mathcal E}\frac{\tau^2}{2}\norm{\frac{\frac{\eta-\eta_{k-1}}{\tau}-\frac{\eta_{k-1}-\eta_{k-2}}{\tau}}{\tau}
}_{L^2}^2+E(\eta) -\langle f_k , \eta\rangle_{L^2},
\end{equation}  
where $\mathcal E$ is defined in \eqref{eq:E}.
Then there exists a $\tau_0>0$ depending on the initial data and the right hand side, such that for all $0<\tau\leq\tau_0$, the stability estimate 
\begin{equation*}
E(\eta_k) +\frac12\norm{\frac{\eta_k-\eta_{k-1}}{\tau}}_{L^2}^2  \leq  \left(E(\eta_0)+\frac12 \norm{\eta_*}_{L^2}^2+ \frac12\norm{f}_{L^2((0,T);L^2)}^2\right)(1+4C\tau h \ell),
\end{equation*}
is satisfied
whenever the right hand side is finite and $(\eta_k)_{k=1,\dots,\floor{T/\tau}}$ does not reach a collision. The constant $C$ depends on the given data only.
\end{theorem}

In the case of a linear higher order term we prove that solutions are unique and hence the scheme converges (for arbitrary initial data). If the initial data is smooth we deduce that the scheme converges with a rate. This result is valid for both schemes, the direct approach as well as the approach with artificial viscosity.
\begin{theorem}[Convergence rate]
\label{thm:conv}
Let $E(\eta)=E_1(\eta)+E_2(\eta)$, where $E_1$ is given by \eqref{kelvinVoigt} and $E_2$ by \eqref{e2:lin} satisfying  \eqref{eqn:alpha0}. Let further $\eta$ be the solution of \eqref{eq:intro} with boundary values \eqref{eq:bv1}, \eqref{eq:bv}. We denote the error $e_k=\eta_k-\eta(\tau k)$, where $\eta_k$ is the variational approximation obtained by step-wise minimization defined by \eqref{def:eta-minimizing-movements} or \eqref{def:eta-minimizing-movements-nodiss}.
There exist constants $C_1,C$ and $\tau_0>0$ depending on the given data, such that for all $0<\tau\leq \tau_0$ the following convergence estimate holds 
\begin{equation*}
\frac12 \norm{\nabla^{k_0}e_k}_{L^2}^2 + \frac{1}{2}  \norm{\frac{e_k-e_{k-1}}{\tau}}_{L^2}^2 \leq \tau^2 CT e^{C_1 k \tau },
\end{equation*}
whenever 
\begin{equation*}
\label{eq:smooth}
\eta_0,\eta_*\in W^{3k_0,2}(Q), \quad f\in W^{2,2}((0,T);L^2(Q)).
\end{equation*}
\end{theorem}
As a by-product we show higher regularity for the $W^{k_0,2}$-case, see \Cref{thm:higher-space-regularity}.

\begin{theorem}[Regularity]
Under precisely the same assumptions as in \Cref{thm:conv}, we find
\begin{equation*}
  \partial_{t} \eta\in L^\infty(((0,T);W^{2k_0,2}_{\text{loc}}(Q)),\quad \partial_{tt} \eta\in L^\infty(((0,T);W^{k_0,2}(Q)),\quad \partial_{ttt} \eta\in L^\infty((0,T);L^2(Q))
\end{equation*}
and $\Delta^{k_0}\partial_t\eta\in L^\infty((0,T);L^2(Q))$,
with natural bounds.
\end{theorem}
Please observe that the rate shown here is the same as is known to be optimal in the case of parabolic evolutions with convex energies. For the sake of completeness we include some simple numeric experiment in \Cref{sec:numerics} that shows the optimality of the rates shown here.

The difference of the hyperbolic to the parabolic case is that more regularity for the initial state and external forcing has to be assumed, as no smoothing effect over time is available. 

In conclusion the variational time-discrete approximation possesses very strong properties regarding the question of stability and convergence. Indeed, for a family of non-convex and non-linear settings, which are in some sense of lower order, the respective hyperbolic evolution has the same stability and/or convergence properties as in the convex and/or linear case, provided the term of leading order is convex and/or linear. 
\begin{remark}[Evolutions with dissipation potential]
The stability and the convergence estimate are both valid in the presence of dissipation under rather general assumptions, which can be checked easily using the methods introduced in this paper. We decided to put our focus on pure hyperbolic motions here, since even so hyperbolic motions are physically very relevant, much less analysis is available in that regime. Indeed, surveying the state of the art in elastodynamics including large deformation shows that in most analytic results a dissipation potential for the elastic deformation is assumed. This was in particular the case in the recent existence result allowing for inertia~\cite[Section 3]{BenesovaKampschulteSchwarzacher}.
\end{remark}

\begin{remark}[Relevance for existence theory]
\label{rem:existence}
The stability and regularity result seem to have the potential for future use in the analysis. Already in the existence theory obtaining a {\em hyperbolic} energy estimate that is given here in form of a stability estimate {\em on the $\tau$-level} was previously not known and allows to circumvent the testing of $\partial_t\eta$ on the $h$-level as was done in the previous literature of the method~\cite{Benesova2023,BenesovaKampschulteSchwarzacher}.

Reversively the regularity theory allows to show that $\partial_t\eta$ is also a test-function a-posteriori, which allows for precise uniqueness and/or the quantification of distances between solutions.
\end{remark}

\subsection{Structure of the paper}

The paper is of two main parts and a supplement with some numeric experiments.

\Cref{sec:problem-setting} and  is about the stability question. Here we first clarify the abstract assumptions necessary for the stability in \ref{ssec:generalas} and the particular assumptions for elastodynamics~\ref{ssec:aselas} of which we show that they are satisfied for the leading elastic example~\ref{ssec:proto}. Second we provide some general Gronwall inequality for discrete schemes with two scales, which will be used to show the stability in~\ref{ssec:gronwall}. This builds one of the technical highlights of this part, another is the critical non-convexity estimate for elastic energies \ref{ssec:non-convex}. The scheme is introduced in \ref{ssec:scheme} and
stability estimate is then proved in \ref{ssec:stab}--\ref{ssec:stabend}.

In \Cref{sec:convergence} we focus on the case of elasto-dynamics for which we explicitly prove regularity and convergence, we also provide abstract assumptions at the beginning of the section. In \Cref{ssec:timereg} we show in-time  regularity for hyperbolic elastodynamics. Theorem~\ref{thm:convergence-rate-solid} is then proved at the end of the section.

We conclude the paper with some numerical experiments in \Cref{sec:numerics} that imply that the rates are optimal, that the appearance of $\tau_0$ in the stability and convergence results is necessary and the loss of convergence in case $h$ and $\tau$ differ. {The experiment is merely to show that our results are in coherence with expectations from the ODE theory. We leave the implementation for largely deforming solids to  a future work.}

\section{Stability}\label{sec:problem-setting}

\subsection{General setting for stability}
\label{ssec:generalas}
We formulate our problem of interest in a rather general case. Let $X$ be a reflexive Banach space, which densely embeds into some Hilbert space $H$. Further, let $\mathcal{E}\subset X$ be a weakly closed subset. Let the energy $E\colon X\to (-\infty,\infty]$ satisfy \Cref{assumptions-general} below.
We consider the problem 
\begin{align}
\begin{aligned}\label{eqn:general-eqn}
\partial_{tt}\eta +DE(\eta) &= f \quad \text{in }(0,T)\\
\eta(t)&\in \mathcal E \text{ for a.e. }t\in(0,T)\\
\eta(0)&=\eta_0\\
\partial_t\eta(0)&=\eta_*,
\end{aligned}
\end{align}
where $DE$ denotes the (Fr\'echet) derivative of $E$.
We assume that the initial conditions and right hand side have
\begin{equation*}
\eta_0\in X, \quad \eta_* \in H, \quad f\in L^2((0,T);H).
\end{equation*}
For convenience we denote \begin{equation}\label{eqn:Linfinity-E}
L^\infty((0,T);\mathcal E)= \{\eta\in L^\infty((0,T);X): \eta(t)\in \mathcal E \text{ for a.e. }t\in(0,T)\}.
\end{equation}

Now we define the notion of weak solution to the problem above.
\begin{definition}[Weak solution]\label{def:weaksol-general}
We say that 
\begin{equation*}
\eta\in L^\infty((0,T);\mathcal E)\quad \text{with}\quad \norm{E(\eta)}_{L^\infty((0,T))}<\infty \quad\text{and}\quad\partial_t\eta\in L^\infty((0,T);H)
\end{equation*}
is a weak solution to \eqref{eqn:general-eqn} if\footnote{Note that since $\eta\in W^{1,\infty}((0,T);H)$, the value $\eta(0)\in H$ is well-defined.} $\eta(0)=\eta_0$ and
\begin{equation*}
\langle \eta_*, \varphi(0)\rangle_H +\int_0^T -\langle\partial_t \eta, \partial_t \varphi \rangle_H +\langle DE(\eta), \varphi\rangle_X\dt = \int_0^T \langle f, \varphi\rangle_H\dt
\end{equation*}
for all $\varphi\in L^2((0,T);X)\cap W^{1,2}((0,T);H)$ with $\varphi(T)=0$.
\end{definition}

\begin{assumptions}\label{assumptions-general}
The energy $E\colon X\to (-\infty,\infty]$ satisfies the following assumptions:
\begin{enumerate}[label=(A.\arabic*), ref=A.\arabic*]
\item $E(\eta)<\infty$ for $\eta\in \mathcal{E}$.\label{as:gen-finite}
\item There is $E_{\min} >\infty$ such that $E(\eta)\geq E_{\min}$ for all $\eta\in X$.\label{as:gen-bdd-below}
\item $E$ is Fr\'echet differentiable at each $\eta\in \operatorname{int} \mathcal{E}$, and $DE\colon \operatorname{int} \mathcal{E} \to X^*$ is strongly continuous.\label{as:gen-E-differentiable}
\item $E$ is coercive in the following sense: for every $E_{\min}\leq K<\infty$ the sublevel set $\{\eta\in X: E(\eta)\leq K\}$ is bounded in $X$.\label{as:gen-coercive}
\item \label{as:nonconvexity-general-Z} There is a Banach space $Z$, such that $X\subset Z$, and a linear operator $L\colon Z\to Z^*$ which is bounded, symmetric and elliptic (meaning that for some $\lambda>0$ it holds $\langle Lz, z\rangle \geq \lambda \|z\|_{Z}^2$ for all $z\in Z$), such that $E$ satisfies the \emph{non-convexity estimate}: For every $E_{\min}\leq K<\infty$ there exists $C$ depending on $K$ such that 
\begin{equation*}
\langle DE(\eta_1), \eta_1-\eta_0\rangle \geq E(\eta_1)-E(\eta_0) -C \norm{\eta_1-\eta_0}_Z^2
\end{equation*}
for all $\eta_1,\eta_0\in \mathcal E$ with $E(\eta_1),E(\eta_0) \leq K$.
\end{enumerate}
\end{assumptions}
The non-convexity estimate is exactly what will enable us to prove the stability of our approximation. Important is its relation to $\mathcal{E}$, which can here be any set. 

The operator $L$ will be used to produce an artificial stabilization in the approximation~\eqref{def:eta-minimizing-movements-gen}. In this setting it will be enough to have the non convexity estimate in $Z$. In case we do not include an extra stabilization term, we can obtain a stability estimate (see \ref{thm:stability-solid}) under the stronger assumption 
\begin{enumerate}[label=(A.5'),ref=A.5']
\item\label{as:nonconvexity-general-H} $E$ satisfies the following \emph{non-convexity estimate}: For every $E_{\min}\leq K<\infty$ there exists $C$ depending on $K$ such that 
\begin{equation*}
\langle DE(\eta_1), \eta_1-\eta_0\rangle \geq E(\eta_1)-E(\eta_0) -C \norm{\eta_1-\eta_0}_H^2
\end{equation*}
for all $\eta_1,\eta_0\in \mathcal E$ with $E(\eta_1),E(\eta_0) \leq K$.
\end{enumerate}

\subsection{Setting for elastodynamics}
\label{ssec:aselas}
As the most prominent application we present the case of dynamic evolution of an elastic solid. The solid is described in Lagrangian coordinates. This means that there is a bounded Lipschitz reference domain $Q\subset \R^n$, and at each time $t$ the solid is described by a deformation $\eta(t)\colon Q\to\R^n$. Then we seek $\eta\colon (0,T)\times Q\to \R^n$ a solution to the equation
\begin{align}
\begin{aligned}\label{eqn:elastic-eqn}
\partial_{tt}\eta+ DE(\eta)&=f\in L^2((0,T);L^2(Q)),  \\
\eta(0)&=\eta_0 \in \mathcal{E},\\ \partial_t\eta(0)&=\eta_*\in L^2(Q), \\
\eta(t,x)&=x, \quad x\in \Gamma_D, \\
\partial_\nu \eta(t,x) &= \nu,\quad x\in \Gamma_N
\end{aligned}
\end{align}
where, $\Gamma_D\cup \Gamma_N=\partial Q$ and $|\Gamma_D|_{d-1}>0$. Here we denote by $\nu$ the outer normal to $\partial Q$ resp. by $\partial_\nu$ the corresponding normal derivative. Further boundary conditions for higher order derivatives naturally appear depending on $E_2$. For more details see the remark on boundary conditions at the end of this subsection.

In any case the set of admissible deformations is 
\begin{equation}
\label{eq:E}
\mathcal E = \left\{ \eta\in X: |\eta(Q)|=\int_Q\det\nabla \eta \dx ,\quad \eta|_{\Gamma_D}(x) = x \text{ for }x\in\Gamma_D \right\},
\end{equation}
where the function space $X$ is one of the following two cases:
\begin{equation}
X=W^{2,q}(Q)\quad \text{or}\quad X=W^{k_0,2}(Q).
\end{equation}
We will henceforth refer to the former as \emph{the $W^{2,q}$-case} and to the latter as \emph{the $W^{k_0,2}$-case}. Regarding the exponents, we assume either 
\begin{equation}
\label{eq:exp1}
q>n
\quad \text{or}\quad
k_0>n/2+1.
\end{equation}
In both cases we have the compact embedding $X\subset\subset C^{1,\alpha}(Q)$ for $0<\alpha<\min (1,\alpha_0)$, where either 
\begin{equation}\label{eqn:alpha0}
\alpha_0=1-n/q\quad \text{or}\quad \alpha_0=k_0-1-n/2.
\end{equation}

The condition $|\eta(Q)|=\int_Q\det\nabla \eta \dx$ is called the Ciarlet--Ne\v cas condition and it guarantees global interior injectivity of $\eta$, cf. \cite{Ciarlet1987}. We can readily see that $\mathcal{E}$ is weakly closed, as the condition is stable under weak convergence in $X$.


Note that in particular if $n=2$ or $n=3$, in the $W^{k_0,2}$-case it is enough to choose $k_0\geq 3$.

Recall the notation \eqref{eqn:Linfinity-E} and also denote\begin{equation}
W^{2,2}_D (Q)= \{u\in W^{2,2}(Q): u|_{\Gamma_D}=0\}
\end{equation}

Now we define the weak solutions, consistently with \Cref{def:weaksol-general}.

\begin{definition}\label{def:weaksol-solid}
We say that $\eta$ with
\begin{equation}\label{eta-weaksol-spaces}
\eta\in L^\infty((0,T);\mathcal E)  \quad\text{with}\quad \norm{ E(\eta)}_{L^\infty((0,T))}<\infty \quad\text{and}\quad \partial_t \eta \in L^\infty((0,T);L^2(Q)).
\end{equation}
is a weak solution to \eqref{eqn:elastic-eqn} if 
\begin{equation*}
\int_0^T -\langle \partial_t \eta ,\partial_t \phi \rangle + \langle DE(\eta), \phi \rangle \dt + \langle \eta_*, \phi(0)\rangle = \int_0^T \langle f, \phi\rangle \dt
\end{equation*}
for all $\phi\in C^\infty([0,T];C^\infty(Q;\R^n))$ with $\phi |_{(0,T)\times \Gamma_D}=0$ and $\phi(T)=0$.
\end{definition}

Now let us specify the assumptions of the elastic energy.
\begin{assumptions}\label{assumptions-energy-solid}
We assume the \emph{energy} $E\colon X\to (-\infty, \infty]$  can be written as the sum 
\begin{equation*}
E(\eta)=E_1(\eta)+E_2(\eta)
\end{equation*}
and $E_1$, $E_2$ satisfy the following assumptions. There exists a density $e$ for $E_1$, that is $e\colon \R^{n\times n} \to (-\infty,\infty]$, such that $E_1$ is of the form
\begin{equation*}
E_1(\eta)=\int_Q e(\nabla\eta) \dx, \quad \eta \in X,
\end{equation*}
 and moreover it holds:
\begin{enumerate}[label=(E.\arabic*), ref=E.\arabic*]
\item $e\in {C}^2(\R^{n\times n}_{\det > 0})$, where $\R^{n\times n}_{\det > 0}= \{M\in\R^{n\times n }: \det M > 0\}$. \label{as:e-C2}
\item There is $e_{\min}>-\infty$ such that $e(\xi)\geq e_{\min}$ for all $\xi\in\R^{n\times n}$.\label{as:e-bdd-below}
\item \label{as:det-lower-bound} For all $K<\infty$ there exists $\epsilon_0>0$ such that for each $\eta\in X$, $E(\eta)\leq K$ implies $\det \nabla\eta \geq \epsilon_0$ in $Q$. 
\item For $\xi\in\R^{n\times n}_{\det>0}$ with $\det \xi\to 0$ it holds $e(\xi)\to\infty$, and $e(\xi)=\infty \text{ for }\det \xi\leq 0$.\label{as:penalize-compression}
\item \label{as:E2-convex} $E_2$ is convex, coercive and differentiable on $X$.
\end{enumerate}
Analogously to the general case, we include separately the stronger convexity assumption that allows to use the approximation without extra stabilization term.
\begin{enumerate}[label=(E.5'),ref=E.5']
\item \label{as:E2-unif-convex} $E_2$ is uniformly convex on $W^{2,2}(Q)$. This means there exists $c>0$ so that it holds for all $\eta\in X$ and $w\in X$
\begin{equation*}
\langle D^2 E_2 (\eta), w\otimes w\rangle \geq c\norm{\nabla^2 w}_{L^2(Q)}^2.
\end{equation*}
\end{enumerate}
\end{assumptions}
Note that in contrast to the abstract setting, we do not assume the non-convexity estimate \eqref{as:nonconvexity-general-Z}. In fact, it will be proven in \Cref{thm:E-noncovexity} that this estimate follows from the other properties and \eqref{as:E2-convex}.
\begin{remark}[Notation]
To avoid confusion with derivatives, we denote the gradient of $\eta\colon Q\to \R^n$ with respect to $x\in Q$ by $\nabla$ (or $\nabla_x$), whereas the gradient of $e\colon \R^{n\times n}\to \R$ will be denoted by $\nabla_\xi$. Similarly for higher derivatives. Moreover, the derivative of $E$ (resp. $E_1$ or $E_2$) will be denoted by $D$, to emphasize that it is a derivative of a functional on the infinite-dimensional space $X$.
\end{remark}
\begin{remark}[More general boundary conditions]
For simplicity we take throughout the paper the assumption that we have boundary conditions
 \begin{equation}
 \label{eq:bv1} 
 g(x)=x
 \end{equation}
 as otherwise the estimates do not change significantly but are much harder to follow. In the case of general boundary function $g$ that is assumed to be extended to $Q$ by $G$ in an appropriate sense (here it means in some Sobolev space $W^{k,p}(Q)$ and with strictly positive Jacobian) we can define
\begin{align}
\label{eq:bv}
\eta(t,x)=g(x)\text{ for } (t,x)\in [0,T]\times \Gamma_D\text{ and }\partial_\nu (\eta(t,x)-G(x))=0\text{ on }[0,T]\times \Gamma_N.
\end{align}
 The related testing space to these boundary conditions is
\[
W=\{\phi\in W^{1,1}(Q)\,:\, \phi(x)=0\text{ for }x\in \Gamma_D\}.
\]
  For the higher order derivatives the respective boundary conditions of Navier type are defined via
 $
 E_2:W^{k,p}(Q)\to [0,\infty],
 $ such that
 \[
 DE_2(\eta-G)\in (W^{k,p}\cap W)^*
 \] 
 In the example of the $k$-Laplacian $E_2(\eta)=\frac{1}{2}\int_Q |\nabla^k \eta|^2 \dx$ this becomes
 \[
\int_Q \nabla^k(\eta-G)\cdot \nabla^k\phi\, dx= \int_Q (-\Delta)^k (\eta-G) \cdot \phi\, dx,
 \]
 which means additional to \eqref{eq:bv} that
 \begin{align*}
 \begin{aligned}
  \partial_\nu \Delta^{k-1}(\eta-G)&=0\text{ on }\Gamma_N,
   \\
   \partial_\nu \nabla^{k-l-1}\Delta^{l}(\eta-G)&=0 \text{ on }\partial Q\text{ for }l\in \{0,...,k-2\}
   \end{aligned}
 \end{align*}
\end{remark}
\subsection{Non-convexity estimate for elastic solids}
\label{ssec:non-convex}
As was indicated already, the non-convexity estimate \eqref{as:nonconvexity-general-Z} is essential for the stability estimates. Previous versions of this estimate in the literature do not allow the estimate for general distances, see~\cite[Proposition 3.2]{MieRou20}. 

As a first step, we show that $e$ and its derivatives are uniformly bounded, with bound depending on the energy only.

\begin{lemma}\label{lem:e-density-bounded}
Let $K\in \R$. Then there exists a $C_K\in \R$ such that every $\eta$ belonging to the weak solution class \eqref{eta-weaksol-spaces} with $\norm{E(\eta)}_{L^\infty((0,T))} \leq K$ satisfies
\begin{equation*}
e(\nabla \eta),|\nabla_\xi e(\nabla \eta)|,|\nabla_\xi^2 e(\nabla \eta)| \leq C_K, \quad \text{in }(0,T)\times Q.
\end{equation*}
\end{lemma}
\begin{proof}
Using the assumption \eqref{as:det-lower-bound} we see that 
\begin{equation*}
\det \nabla \eta \geq \epsilon_0 \quad \text{in }(0,T)\times Q
\end{equation*}
with $\epsilon_0$ depending only on $K$.
Moreover, since $X$ is embedded into $C^{1,\infty}$, we see that, by boundedness of $E_2$, $|\nabla \eta (t,x)| \leq C$, where $C$ depends on $K$ and the embedding $X\subset C^{1,\infty}$. This means that
\begin{equation*}
\nabla \eta(t,x)\in \mathcal{K} := \{ M\in\R^{n\times n} : \det M\geq \epsilon_0, |M| \leq C \} .
\end{equation*}
Since $\mathcal{K}$ is a compact set contained in $\R^{n\times n}_{\det >0}$ where $e$ is $C^2$, we know that $e,\nabla_\xi e,\nabla_\xi^2 e$ are bounded on $\mathcal K$. Since $\mathcal K$ depends only on $K$, the proof is finished.
\end{proof}

Now we will estimate the non-convexity of $E_1$ in terms of the distance of gradients. 
\begin{lemma}[Non-convexity estimate (I)]\label{thm:E-noncovexity}
Suppose that we have $\eta_0,\eta_1\in\mathcal E$ with	
\begin{equation*}
E(\eta_0),E(\eta_1)\leq K <\infty.
\end{equation*}
Then there exists a constant $C_1$ depending only on $ K$ such that 
\begin{equation*}
\langle DE_1(\eta_1), \eta_1-\eta_0\rangle \geq E_1(\eta_1)-E_1(\eta_0)  - C_1 \norm{\nabla \eta_1-\nabla\eta_0}_{L^2}^2.
\end{equation*}
\end{lemma}
\begin{proof}
Throughout the proof, any constant named $C_i$ with any index $i$ depends only on $K$.
By coercivity of $E$, we have that $\norm{\eta_0}_X,\norm{\eta_1}_X\leq C_K$ and thus $\norm{\nabla \eta_0}_{L^\infty},\norm{\nabla \eta_1}_{L^\infty}\leq C_\infty$. 
Let $\epsilon_0$ be the lower bound on the determinant from \eqref{as:det-lower-bound}, corresponding to $K$. Now the set 
\begin{equation*}
\mathcal{A} = \{M\in\R^{n\times n} : \det M \geq \epsilon_0, |M|\leq C_\infty\}
\end{equation*}
is compact in $\R^{n\times n}_{\det > 0}$, and
\begin{equation*}
\mathcal{B} = \{M\in \R^{n\times n} : \det M \geq \epsilon_0/2, |M| \leq 2C_\infty \}
\end{equation*}
is likewise compact in $\R^{n\times n}_{\det > 0}$, with $\mathcal A\subset \operatorname{int} \mathcal B$. So there exists $r>0$ such that $B_r(\mathcal A)\subset \mathcal B$, in other words,
$
\text{ for all } A\in \mathcal A  \text{ and for all } B\in \R^{n\times n}$, $|B-A|<r$ implies $ B\in \mathcal B.
$

Now let us split our expression in two parts:
\begin{equation*}
\langle DE_1(\eta_1), \eta_1 -\eta_0\rangle = \int_{\{|\nabla \eta_1-\nabla \eta_0| \leq r\}} \hspace{-2em}\nabla_\xi e(\nabla \eta_1): \nabla (\eta_1-\eta_0) \dx  +\int_{\{|\nabla \eta_1-\nabla \eta_0| > r\}} \hspace{-2em}\nabla_\xi e(\nabla \eta_1): \nabla (\eta_1-\eta_0) \dx.
\end{equation*}
For the second part, we recall from Lemma \ref{lem:e-density-bounded} that $|\nabla_\xi e(\nabla \eta_1) |\leq C_K$ and we get
\begin{equation*}
\int_{\{|\nabla \eta_1-\nabla \eta_0| > r\}} \nabla_\xi e(\nabla \eta_1): \nabla (\eta_1-\eta_0) \dx \geq -\frac{C_K}{r} \norm{\nabla \eta_1-\nabla\eta_0}_{L^2}^2.
\end{equation*}
Let us now for $\theta\in [0,1]$ denote $\eta_\theta=\theta \eta_1+(1-\theta)\eta_0$. Then for the first term, we apply pointwisely (i.e. for each $x$) the Taylor theorem of the second order (with respect to $\theta$). So we obtain that there exists $\theta\colon Q \to [0,1]$ such that \begin{align*}
\int_{\{|\nabla \eta_1-\nabla \eta_0| \leq r\}} \hspace{-5em}\nabla_\xi e(\nabla \eta_1): \nabla (\eta_1-\eta_0) \dx =\int_{\{|\nabla \eta_1-\nabla \eta_0| \leq r\}}\hspace{-5em} e(\nabla \eta_1) -e(\nabla \eta_0) +\nabla^2_\xi e(\nabla \eta_\theta)(\nabla\eta_1-\nabla\eta_0,\nabla\eta_1-\nabla\eta_0) \dx \\ = E_1(\eta_1)-E_1(\eta_0)-\int_{\{|\nabla \eta_1-\nabla \eta_0| > r\}}\hspace{-5em} e(\nabla\eta_1) -  e(\nabla\eta_0)\dx+
\int_{\{|\nabla \eta_1-\nabla \eta_0| \leq r\}} \hspace{-5em}\nabla^2_\xi e(\nabla \eta_\theta)(\nabla\eta_1-\nabla\eta_0,\nabla\eta_1-\nabla\eta_0) \dx 
\end{align*}
The middle term we estimate similarly as before
\begin{equation*}
-\int_{\{|\nabla \eta_1-\nabla \eta_0| > r\}}  e(\nabla\eta_1) -  e(\nabla\eta_0)\dx \geq -2\frac{C_K}{r^2}\norm{\nabla\eta_1-\nabla\eta_0}^2.
\end{equation*}
It thus remains to estimate the last quadratic term, namely it suffices to show that 
\begin{equation*}
|\nabla^2_\xi e(\nabla\eta_\theta)| \leq C_2\quad \text{on} \quad \{|\nabla\eta_1-\nabla \eta_0|\leq r\}.
\end{equation*}
Due to our assumption, we have $\nabla \eta_0(x),\nabla\eta_1(x)\in\mathcal A$ for all $x\in Q$. Therefore if $|\nabla\eta_1(x)-\nabla\eta_0(x)|\leq r$, then $\nabla\eta_{\theta(x)}(x)\in \mathcal{B}$. So the inequality holds with $C_2=\max_{\mathcal{B}} |\nabla^2_\xi e|$, which is finite due to $e\in C^2(\R^{n\times n}_{\det >0})$ and $\mathcal B$ being compact in $\R^{n\times n}_{\det >0}$. Putting together all the inequalities, we see that we proved our claim with $C_1=\frac{C_K}{r}+2\frac{C_K}{r^2}+C_2$.
\end{proof}

\begin{lemma}\label{EX-implies-AX}
\Cref{assumptions-energy-solid}, \eqref{as:e-C2}--\eqref{as:E2-convex} imply \Cref{assumptions-general}, \eqref{as:gen-finite}--\eqref{as:nonconvexity-general-Z}.
\end{lemma}
\begin{proof}
It is readily seen that \eqref{as:e-C2} implies \eqref{as:gen-finite} and \eqref{as:e-bdd-below} implies \eqref{as:gen-bdd-below}. The Fr\'echet differentiability  of $E_1$ on $\operatorname{int}\mathcal E$ follows from \eqref{as:e-C2} combined with \eqref{as:det-lower-bound}, as then we can see that the derivative at $\eta\in \operatorname{int}\mathcal E$ in the direction $\gamma\in X$ is then
\begin{equation*}
\langle DE_1(\eta), \gamma\rangle = \int_Q \nabla_\xi e(\nabla\eta): \nabla \gamma \dx,
\end{equation*}
differentiability of $E_2$ is already assumed in \eqref{as:E2-convex}. This shows \eqref{as:gen-E-differentiable}. Coercivity \eqref{as:gen-coercive} is thanks to \eqref{as:E2-convex}  and the fixed boundary values on $\Gamma_D$.
The non-convexity estimate \eqref{as:nonconvexity-general-Z} follows from \Cref{thm:E-noncovexity} and \eqref{as:E2-convex}.
\end{proof}

In case no stabilizer is considered the argument needs to be refined, which is done by Lemma~\ref{thm:E-noncovexityii}: By interpolation and using \eqref{as:E2-unif-convex}, the $W^{2,2}$-uniform convexity of $E_2$, we can estimate the non-convexity of $E$ in terms of $L^2$ distance only. 

\begin{lemma}[Non-convexity estimate (II)]\label{thm:E-noncovexityii}
Let $E_2$ satisfy additionally \eqref{as:E2-unif-convex} and suppose that we have $\eta_0,\eta_1\in\mathcal E$ with	
\begin{equation*}
E(\eta_0),E(\eta_1)\leq K <\infty.
\end{equation*}
Then there exists a constant $C$ depending only on $K$ such that 
\begin{equation*}
\langle DE(\eta_1), \eta_1-\eta_0\rangle \geq E(\eta_1)-E(\eta_0) - C\norm{\eta_1-\eta_0}_{L^2}^2.
\end{equation*}
\end{lemma}
\begin{proof}
Using Taylor theorem to the second order to get for some $\xi\in[0,1]$, $\eta_\xi= \xi \eta_1 + (1-\xi)\eta_0$
\begin{multline*}
\langle DE_2(\eta_1),\eta_1-\eta_0\rangle = E_2(\eta_1)-E_2(\eta_0) + \frac12 \langle D^2E_2(\eta_\xi), (\eta_1-\eta_0)\otimes (\eta_1-\eta_0)\rangle \\
\geq E_2(\eta_1)-E_2(\eta_0)+ c\norm{\nabla^2(\eta_1-\eta_0)}_{L^2}^2
\end{multline*}
Interpolate 
\begin{equation*}
C_1\norm{\nabla\eta_1-\nabla\eta_0}_{L^2}^2 \leq c\norm{\nabla^2\eta_1-\nabla^2\eta_0}_{L^2}^2 + C\norm{\eta_1-\eta_0}_{L^2}^2.
\end{equation*}
Combining these two inequalities and plugging this into the result of \Cref{thm:E-noncovexity} gives the desired result.
\end{proof}

\begin{lemma}
\Cref{assumptions-energy-solid} \eqref{as:e-C2}-\eqref{as:penalize-compression}, \eqref{as:E2-unif-convex} imply \Cref{assumptions-general} \eqref{as:gen-finite}-\eqref{as:gen-coercive}, \eqref{as:nonconvexity-general-H}.
\end{lemma}
\begin{proof}
The non-convexity estimate \eqref{as:nonconvexity-general-H} follows from \Cref{thm:E-noncovexity}. The validity of \eqref{as:gen-finite}-\eqref{as:gen-coercive} has already been shown in \Cref{EX-implies-AX}.
\end{proof}

Let us now in the equation \eqref{eqn:elastic-eqn} rewrite the term $DE_1(\eta)$ in terms of $e$. Assume $\eta$ lies in the spaces \eqref{eta-weaksol-spaces} and compute the Gateaux derivative at $\eta(t)\in \mathcal E$ in a direction $\gamma\in X$. Recall that $X\subset W^{2,2}(Q)$ so we can use the chain rule to obtain at time $t$
\begin{align*}
\langle DE_1(\eta(t)),\gamma\rangle =  \int  \nabla_\xi e(\nabla_x \eta):\nabla_x \gamma =\int_Q\sum_{i,j=1}^n \partial_{\xi_j^i} e(\nabla\eta) \partial_{x_i} \gamma  = -\int_Q \sum_{i,j=1}^n \partial_{x_i}( \partial_{\xi^i_j} e( \nabla \eta))\gamma_j\\=- \int_Q \sum_{i,j,k,l=1}^n  \partial^2_{\xi^i_j\xi^k_l} e( \nabla \eta)\partial^2_{x_i x_l}\eta_k\gamma_j=-\int_Q \nabla^2_\xi e(\nabla_x \eta) : \nabla^2_x \eta \cdot \gamma.
\end{align*}
then by the previous \Cref{lem:e-density-bounded} we see $\nabla_\xi^2e(\nabla\eta)\in L^\infty((0,T)\times Q)$ and $\nabla^2 \eta\in L^2((0,T)\times Q)$, so then $DE_1(\eta)=-\nabla^2_\xi e(\nabla_x \eta) : \nabla^2_x \eta \in L^2((0,T)\times Q)$ in the usual sense.

So our then equation \eqref{eqn:elastic-eqn} can be written as
\begin{equation*}
\partial_{tt}\eta+DE_2(\eta)- \nabla^2_\xi e(\nabla_x \eta) : \nabla^2_x \eta = 0.
\end{equation*}


\subsection{Prototypical energy}
\label{ssec:proto}
As a prototype of the highest-order convex regularizing part of the energy we can put
\begin{equation}
\text{if }X=W^{k_0,2}(Q): \quad E_2(\eta)=\frac12 \norm{\nabla^{k_0}\eta}_{L^2(Q)}^2=\frac12\int_{Q}\left |\nabla^{k_0}\eta\right |^2\dx,
\end{equation}
or 
\begin{equation}
\text{if }X=W^{2,q}(Q):\quad  E_2(\eta)=\frac{1}{q}\int_Q (1+|\nabla^2\eta|)^{q-2}|\nabla^2\eta|^2 \dx.
\end{equation}
They are both uniformly convex thus satisfy \eqref{as:E2-unif-convex}, the latter thanks to $t\mapsto (1+|t|)^{q-2}|t|^q$ being uniformly convex.
Further, the $q$-biLaplacian \begin{equation}\label{q-bilaplacian}
E_2(\eta)=\frac1q \norm{\nabla^2 \eta}_{L^q(Q)}^q = \frac1q\int_Q |\nabla^2\eta|^q \dx
\end{equation}
satisfies the convexity \eqref{as:E2-convex}, but not the uniform convexity \eqref{as:E2-unif-convex}.

As there is no difference in the analysis we simplify the physical energy \eqref{kelvinVoigt} to its determinant part and use as prototype of the energy density for $E_1$ 
\begin{equation}\label{prototypical-e}
e(\xi)=
\begin{cases}
\displaystyle\frac{1}{(\det \xi)^a}, & \det\xi>0\\
\infty,& \det\xi \leq 0
\end{cases}
\end{equation} 
with $a> n/\alpha_0$ (recall that $\alpha_0$ is defined by \eqref{eqn:alpha0}). It can readily be checked that the next theory also holds for $E_1$ being in the form of \eqref{kelvinVoigt}.
\begin{theorem}
The prototypical energy $E_1$ defined in \eqref{prototypical-e} satisfies \Cref{assumptions-energy-solid} \eqref{as:e-C2}-\eqref{as:penalize-compression}.
\end{theorem}
\begin{proof}
 We only need to check the lower bound on the determinant \eqref{as:det-lower-bound}, since all other properties are clear. 
This is essentially proven in \cite{Healey2009}, but for completeness we give a concise proof here. Let $\eta \in \mathcal{E}$ with $E(\eta) \leq K$ for some given $K<\infty$. Thus we have by the coercivity of $E$ on $X$ a bound on $\eta$ in $X$, thus also a bound on $\norm{\nabla\eta}_{C^{0,\alpha}}$ and therefore $\norm{\det\nabla\eta}_{C^{0,\alpha}}\leq c_\alpha$, where $c_\alpha$ depends on $K$.

Because the boundary of $Q$ is Lipschitz continuous, there exists a constant $c_L>0$ and a $\delta_0>0$, such that for all $\delta\in (0,\delta_0]$ we have for all $x\in \overline{Q}$
\begin{equation*}
|B_\delta(x) \cap Q| \geq  c_L \delta^n.
\end{equation*}

Let $x_0\in \overline{Q}$ be such that $\det \nabla\eta (x_0)= \min_{x\in \overline{Q}}\det\nabla\eta(x)>0$. Put $\epsilon_0 =\min(\det\nabla\eta(x_0),\delta_0^\alpha)$ and take $0<\delta\leq \delta_0$ arbitrary. Therefore, for $x_0\in Q$ such that $\det\nabla\eta(x_0)\geq \epsilon_0$ we have
\begin{align*}
K\geq E(\eta)\geq \int_Q \frac{1}{(\det\nabla\eta)^a}\dx \geq  \int_{B_\delta(x_0)\cap Q} \frac{1}{\det\nabla\eta(x)^a}\dx \\ \geq  \int_{B_\delta(x_0)\cap Q} \frac{1}{(\det\nabla\eta(x_0)-|\det\nabla\eta(x)-\det\nabla\eta(x_0)|)^a}\dx \\ \geq  \int_{B_\delta(x_0)\cap Q} \frac{1}{(\epsilon_0 + c_\alpha \delta^\alpha)^a }\dx \geq c_L\frac{\delta^n}{(\epsilon_0 + c_\alpha \delta^\alpha)^a}.
\end{align*}
As $\delta$ was arbitrary, we can choose $\delta=\epsilon_0^{1/\alpha}$ and obtain
\begin{equation*}
K\geq c_L\frac{\epsilon_0^{n/\alpha}}{(\epsilon_0 + c_\alpha \epsilon_0)^a}= \epsilon_0^{n/\alpha-a} \frac{c_L}{(1+c_\alpha)^a}
\end{equation*}
and since $n/\alpha-a<0$, we obtain 
\begin{equation*}
\det\nabla\eta(x_0)\geq \epsilon_0\geq \left(\frac{K(1+c_\alpha)^a}{c_L}\right)^{\frac{1}{n/\alpha-a}}.
\end{equation*}
Since the right hand side depends only on $K$, this gives the lower bound on $\det \nabla\eta$ and proves \eqref{as:det-lower-bound}.
\end{proof}

\begin{remark}
For the prototypical energy density we can calculate more explicitly the bounds from \Cref{lem:e-density-bounded} as follows.
Directly from \eqref{prototypical-e} see that \begin{equation*}
e(\nabla \eta) \leq \epsilon_0^{-1/a}.
\end{equation*}

Next we compute the first and second gradient of $e$. 
We have, as $\nabla_\xi\det\xi=\cof\xi$,
\begin{equation*}
\nabla_\xi e(\xi)=-\frac{a}{(\det \xi)^{a+1}}\cof \xi,
\end{equation*}
so the lower bound on the determinant \eqref{as:det-lower-bound}, along with the bound on $\norm{\nabla\eta}_{L^\infty}$ suffices to bound $\nabla_\xi e(\nabla \eta)$.

For the second gradient we compute
\begin{equation*}
\nabla^2_\xi (\det\xi)= (\partial_{\xi^k_l}(-1)^{i+j}\det \xi^{\hat i}_{\hat j})_{i,j,k,l}=\left( \begin{cases}
0,& i=k\text{ or }j=l,\\
(-1)^{i+j+k+l}\det \xi^{\widehat{ik}}_{\widehat{jl}}, &\text{else}
\end{cases}  \right)_{i,j,k,l}
\end{equation*}
where $\xi^{\hat i}_{\hat j}$ is the matrix $\xi$ with row $i$ and column $j$ deleted, likewise for $\xi^{\widehat{ik}}_{\widehat{jl}}$ there are rows $i,k$ and columns $j,l$ deleted\footnote{If after deleting we would have $0\times 0$ matrix, this determinant is defined as $1$.}. Therefore for our prototype we have for any $\det\xi>0$
\begin{equation*}
\nabla^2_\xi e(\xi)=\frac{a(a+1)}{(\det \xi)^{a+2}}\cof \xi \otimes\cof \xi +\frac{-a}{(\det\xi)^{a+1}}\nabla_\xi(\cof\xi)
\end{equation*}
or in components after plugging in $\nabla \eta$
\begin{align*}
&\nabla^2_\xi e(\nabla_x \eta) 
\\
&= \! \!\left( \begin{cases}
\frac{a(a+1)}{(\det\nabla\eta)^{a+2}}\det(\nabla\eta)^{\hat i}_{\hat j}\det(\nabla\eta)^{\hat k}_{\hat l} ,\quad i=k\text{ or }j=l,\\
\frac{a(a+1)}{(\det\nabla\eta)^{a+2}}(-1)^{i+j+k+l}\det(\nabla\eta)^{\hat i}_{\hat j}\det(\nabla\eta)^{\hat k}_{\hat l}+
(-1)^{i+j+k+l}\frac{-a}{(\det\nabla\eta)^{a+1} }\det (\nabla\eta)^{\widehat{ik}}_{\widehat{jl}}, &\text{else}
\end{cases}  \right)_{\!\!\! i,j,k,l}
\end{align*}
so again lower bound on $\det \nabla \eta$ and bound on $\norm{\nabla \eta}_{L^\infty}$ suffices to calculate an explicit bound.
\end{remark}

\subsection{Gronwall--type inequalities}\label{ssec:gronwall}
To show the stability and convergence rate of our scheme, we will make use of some inequalities of discrete Gronwall type.
For completeness, we start with the classical version of the Gronwall inequality, including a short proof.

\begin{lemma}[Discrete Gronwall inequality]\label{lem:gronwall-disc}
Let $a_0,\dots,a_n \geq 0$ and $c_0,\dots,c_{n-1}\geq 0$ satisfy
$$a_k\leq a_0+ \sum_{i=0}^{k-1} c_i a_i, \quad k=1,\dots,n.$$
Then
$$a_k\leq a_0\prod_{i=0}^{k-1}(1+c_i) \leq a_0 \exp\left(\sum_{i=0}^{k-1} c_i\right),\quad k=0,\dots,n.$$
\end{lemma}
\begin{proof}
By induction, we prove the stronger inequality 
$$a_0+ \sum_{i=0}^{k-1} c_i a_i\leq a_0 \prod_{i=0}^{k-1}(1+c_i).$$
For $k=1$ both sides are equal to $a_0(1+c_0)$. Then for $k>1$ proceed by induction
$$a_0+ \sum_{i=0}^{k-1} c_i a_i\leq a_0+ \sum_{i=0}^{k-2} c_i a_i +c_{k-1}a_{k-1} \leq  a_0\prod_{i=0}^{k-2}(1+c_i)+ c_{k-1}a_0\prod_{i=0}^{k-2}(1+c_i) = a_0(1+c_{k-1})\prod_{i=0}^{k-2}(1+c_i)  $$
which concludes the proof, using $1+c_i\leq e^{c_i}$ to obtain the second inequality.
\end{proof}

In fact it will be more useful for us to shift the indices by 1.

\begin{lemma}[Discrete Gronwall inequality, shifted $k$]\label{lem:gronwall-c-shiftk}
Let $a_0,\dots,a_n \geq 0$ satisfy with $0\leq c_1,\dots,c_n<1$
$$a_k\leq a_0+\sum_{i=1}^k c_ia_i,\quad  k=1,\dots,n.$$
Then it holds
$$a_k\leq a_0\prod_{i=1}^k(1-c_i)^{-1}, \quad k=1,\dots,n.$$
In particular, if also $c_1,\dots,c_n\leq 1/2$ then
\begin{equation*}
a_k\leq a_0 \exp\left( 2\sum_{i=1}^k c_i \right),\quad k=1,\dots,n.
\end{equation*}
\end{lemma}
\begin{proof}
For $k=1$ we have $a_1\leq a_0+c_1a_1$, so it is enough to subtract $c_1a_1$ and divide by $1-c_1$.

For $k\geq 2$ proceed by induction. Write $$a_k\leq a_0+\sum_{i=1}^k c_ia_i \leq a_0+\sum_{i=1}^{k-1} c_ia_0\prod_{j=1}^i(1-c_j)^{-1}+a_kc_k=a_0+\sum_{i=1}^{k-1} \frac{c_i}{1-c_i} a_0\prod_{j=1}^{i-1}(1-c_j)^{-1}+a_k c_k$$
subtract $c_ka_k$ and use $\frac{c_i}{1-c_i}=(1-c_i)^{-1}-1$ so that we get a telescoping sum
$$a_k(1-c_k)\leq a_0+\sum_{i=1}^{k-1} (1-c_i)^{-1} a_0\prod_{j=1}^{i-1}(1-c_j)^{-1}-\sum_{i=1}^{k-1} a_0\prod_{j=1}^{i-1}(1-c_j)^{-1}=a_0\prod_{i=k}^{k-1}(1-c_i)^{-1} $$
which is what we wanted.

In the case that $c_i\leq 1/2$ we can use $(1-c_i)^{-1}\leq 1+2c_i\leq e^{2c_i}$ to obtain the second inequality.
\end{proof}

\begin{remark}
Notice that the requirement $c_i<1$ is natural here, since for $c_i\geq 1$ the inequality in the assumption does not pose any restriction on $a_i$.
\end{remark}

The following version of Gronwall inequality, including a square-root term, will enable us to get the natural estimate with the forcing term $f$ being present in the equation.

\begin{lemma}[Discrete Gronwall inequality with square root]
Let $a_0,\dots, a_n \geq 0$ satisfy with $c_0,\dots c_{n-1}\geq 0$ and $d_0,\dots d_{n-1}\geq 0$ the inequality
$$a_k\leq a_{k-1} + c_{k-1} a_{k-1}+d_{k-1}\sqrt {a_{k-1}}, \quad k=1,\dots,n .$$
Then it holds 
$$ a_k\leq \left(\sqrt{a_0}+\frac12\sum_{i=0}^{k-1} d_i  \right)^2 \prod_{i=0}^{k-1}(1+c_i)\leq\left(\sqrt{a_0}+\frac12\sum_{i=0}^{k-1} d_i  \right)^2 e^{\sum_{i=0}^{k-1}c_i}, \quad k=1,\dots,n.$$
\end{lemma}
\begin{proof}
By induction on $k$. For $k=1$: $$a_1\leq a_0+c_0a_0+d_0\sqrt {a_0} \leq \left(\sqrt {a_0} +\frac{d_0}{2}\right)^2 +c_0a_0\leq \left(\sqrt {a_0}+\frac{d_0}{2}\right)^2(1+c_0).$$
Now if the inequality holds for $k-1$, we get \begin{align*} a_k&\leq a_{k-1} + c_{k-1} a_{k-1}+d_{k-1}\sqrt {a_{k-1}} \leq \left(\sqrt{a_{k-1}}+\frac{d_{k-1}}{2}\right)^2(1+c_{k-1}) \\ &\leq \left(\left(\sqrt{a_0}+\frac12\sum_{i=0}^{k-2}d_i\right)\prod_{i=0}^{k-2}\sqrt{1+c_i}+\frac12 d_{k-1}\right)^2(1+c_{k-1})\leq \left( \sqrt{a_0}+\frac12\sum_{i=0}^{k-1}d_i\right)^2 \prod_{i=0}^{k-1}(1+c_i),\end{align*}
and we conclude with $1+c_i\leq e^{c_i}$.
\end{proof}

In fact we will need a version with index on the right shifted by $1$, which is due to using an implicit scheme in our approximation.

\begin{lemma}[Discrete Gronwall inequality with square root, shifted $k$]\label{gronwall-sqrt-shiftk}
Let $a_0,\dots, a_n \geq 0$ satisfy with $0\leq c_1,\dots,c_n<1$ and $d_1,\dots d_{n}\geq 0$ the inequality
$$a_k\leq a_{k-1} + c_k a_{k}+d_{k}\sqrt {a_{k}}, \quad k=1,\dots,n .$$
Then it holds $$a_k \leq \left(\sqrt{a_0} + \sum_{i=1}^k \frac{d_i}{\sqrt{1-c_i}}\right)^2 \prod_{i=1}^k (1-c_i)^{-1},\quad k=1,\dots,n.$$
In particular, for $0<c_i\leq 1/2$ we have also
$$a_k \leq \left(\sqrt{a_0} + \sum_{i=1}^k \frac{d_i}{\sqrt{1-c_i}}\right)^2 \exp\left(2\sum_{i=1}^kc_i\right),\quad k=1,\dots,n.$$
\end{lemma}

\begin{proof}
Proceed by induction and assume it holds for $k-1$ (for $k=0$ the inequality is trivially true). Rewrite the inequality $$(1-c_k)a_k-d_k\sqrt{a_k}\leq a_{k-1},$$ so that $$\left(\sqrt{1-c_k}\sqrt{a_k}-\frac{d_k}{2\sqrt{1-c_k}}\right)^2\leq a_{k-1}+\frac{d_k^2}{4(1-c_k)}.$$ From this, we express $a_k$, and use $\sqrt{A+B}\leq \sqrt{A}+\sqrt{B}$ to obtain  $$a_k\leq\left(\sqrt{a_{k-1}+\frac{d_k^2}{4(1-c_k)}}+\frac{d_k}{2\sqrt{1-c_k}}\right)^2\frac{1}{1-c_k}\leq\left(\sqrt{a_{k-1}}+\frac{d_k}{\sqrt{1-c_k}}\right)^2\frac{1}{1-c_k}$$ thus by induction 
\begin{multline*}a_k\leq\left(\left(\sqrt{a_0}+\sum_{i=1}^{k-1}\frac{d_i}{\sqrt{1-c_i}}\right)\prod_{i=1}^{k-1}\sqrt{(1-c_i)^{-1}}+\frac{d_k}{\sqrt{1-c_k}}\right)^2\frac{1}{1-c_k}\\ \leq \left(\sqrt{a_0}+\sum_{i=1}^{k}\frac{d_i}{\sqrt{1-c_i}}\right)^2\prod_{i=1}^k(1-c_i)^{-1} \end{multline*}
which proves the desired inequality. Finally, note that for $0<c_i\leq 1/2$ it holds $(1-c_i)^{-1}\leq 1+2c_i\leq e^{2c_i}$.
\end{proof}

\subsubsection*{Two-scale Gronwall inequalities}
{Here we state the two-scale analogues of \Cref{lem:gronwall-c-shiftk} and \Cref{gronwall-sqrt-shiftk}, respectively. The particular form of the inequality is suitable to estimates of solutions arising from the minimization scheme \eqref{def:eta-minimizing-movements}.

\begin{theorem}[Two--scale Gronwall inequality I]\label{twoscale-gronwall-classic}
Let $M,N\in\N$ and let us have the sequences $a_k^\ell, b_k^\ell, d_k^\ell\geq 0$, $k=0,\dots,N$, $\ell=0,\dots,M-1$ satisfying $a_0^{\ell}=a_N^{\ell-1}$, $b_0^{\ell}=b_N^{\ell-1}$, $\ell=1,\dots,M$. Assume we have for some $0\leq c<1/2N$ the estimate
\begin{equation*}
a_k^\ell + \frac1N b_k^\ell\leq a_{k-1}^\ell+\frac1N b_k^{\ell-1}+ c a_k^\ell + cb_k^\ell +d_k^l, \quad k=1,\dots,N, \, \ell=0,\dots,M-1,
\end{equation*}
where we put $b_k^{-1}:=b_0^0$, $k=1,\dots,N$. Then it holds
\begin{equation*}
\max_{k=1,\dots,N} \left( a_k^\ell + \frac1N\sum_{i=1}^k b_i^\ell\right) \leq \left(a_0^0+b_0^0+\sum_{l=1}^\ell\sum_{k=1}^N d_k^\ell\right)(1-cN)^{-\ell},\quad \ell=1,\dots,M-1.
\end{equation*}
\end{theorem}
\begin{proof}
Let $k_\ell:= \operatorname{argmax}_{k=1,\dots,N} a_k^\ell+\frac1N\sum_{i=1}^k b_i^\ell$. Then we have after summing $1,\dots,k_\ell$
\begin{equation*}
a_{k_\ell}^\ell + \frac1N \sum_{k=1}^{k_\ell}b_k^\ell \leq a_0^\ell + \frac1N\sum_{k=1}^{k_\ell} b_k^{\ell-1} + c \sum_{k=1}^{k_\ell} a_k^\ell + c \sum_{k=1}^{k_\ell}b_k^\ell + \sum_{k=1}^{k_\ell}d_k^\ell
\end{equation*}
Now first remember that $a_0^\ell=a_N^{\ell-1}$, denote $\alpha_\ell:=a_{k_\ell}^\ell + \frac1N \sum_{k=1}^{k_\ell}b_k^\ell$ (so that $\alpha_\ell= \operatorname{max}_{k=1,\dots,N} a_k^\ell+\frac1N\sum_{i=1}^k b_i^\ell$).
Use the inequalities
\begin{align*}
a_N^{\ell-1} + \frac1N\sum_{k=1}^{k_\ell} b_k^{\ell-1} &\leq
a_N^{\ell-1} + \frac1N\sum_{k=1}^{N} b_k^{\ell-1} \leq
a_{k_{\ell-1}}^{\ell-1} + \frac1N\sum_{k=1}^{k_{\ell-1}} b_k^{\ell-1} = \alpha_{\ell-1}  \\
c \sum_{k=1}^{k_\ell} b_k^\ell &\leq cN \alpha_\ell\\
c \sum_{k=1}^{k_\ell} a_k^\ell &\leq c k_\ell \alpha_
\ell \leq cN \alpha_\ell\\
\sum_{k=1}^{k_\ell} d_k^\ell &\leq \sum_{k=1}^N d_k^\ell
\end{align*}
so it becomes 
\begin{equation*}
\alpha_\ell \leq \alpha_{\ell-1} +2cN \alpha_\ell + \sum_{k=1}^N d_k^\ell.
\end{equation*}
Summing this over $\ell$ we get 
\begin{equation*}
\alpha_\ell \leq \alpha_0 +2cN \sum_{l=1}^\ell\alpha_l + \sum_{l=1}^\ell \sum_{k=1}^N d_k^l
\end{equation*}
so applying \Cref{lem:gronwall-c-shiftk} (using in this lemma $a_0$ as $\alpha_0 +\sum_{l=1}^\ell \sum_{k=1}^N d_k^l$) we see that
\begin{equation*}
\alpha_\ell\leq \left(\alpha_0 + \sum_{l=1}^\ell  \sum_{k=1}^N d_k^l\right){(1-cN)^{-\ell}},
\end{equation*}
which finishes the proof, since $\alpha_0=a_0^0+b_0^0$.
\end{proof}

\begin{theorem}[Two-scale Gronwall inequality II] \label{twoscale-gronwall-sqrt} Let $M,N\in\mathbb N$ and let us have the sequences $a_k^\ell, b_k^\ell, d_k^\ell\geq 0$, $k=0,\dots,N$, $\ell=0,\dots,M-1$ satisfying $a_0^{\ell}=a_N^{\ell-1}$, $b_0^{\ell}=b_N^{\ell-1}$, $\ell=1,\dots,M$. Assume we have for some $0\leq c<1/N$ the estimate
$$
a_k^\ell + \frac1N b_k^\ell\leq a_{k-1}^\ell+\frac1N b_k^{\ell-1}+ c b_k^\ell + d_k^\ell \sqrt{b_k^\ell} , \quad k=1,\dots,N, \, \ell=0,\dots,M-1,
$$
where we put $b_k^{-1}:=b_0^0$, $k=1,\dots,N$. Then it holds
$$
\max_{k=1,\dots,N} \left( a_k^\ell + \frac1N\sum_{i=1}^k b_i^\ell\right) \leq \left(\sqrt{a_0^0+b_0^0}+\frac{1}{\sqrt{1-cN}}\sum_{l=1}^\ell\sqrt{N \sum_{k=1}^N (d_k^\ell)^2}\right)^2 (1-cN)^{-\ell}, \ell=1,\dots,M-1.
$$
\end{theorem}
\begin{proof}
Let $k_\ell:=\operatorname{argmax}_{k=1,\dots,N} a_k^\ell+\frac1N\sum_{i=1}^kb_l^\ell$ for $\ell=1,\dots,N$. Then we have after summing over $1,\dots,k_\ell$
$$a_{k_\ell}^\ell+\frac1N\sum_{k=1}^{k_\ell} b_k^\ell \leq a_0^\ell +\frac1N\sum_{k=1}^{k_\ell} b_{k}^{\ell-1}+c\sum_{k=1}^{k_\ell} b_k^\ell + \sum_{k=1}^{k_\ell}d_k^\ell \sqrt{b_k^\ell}.$$
Now denote $\alpha_\ell:=a_{k_\ell}^\ell + \frac1N \sum_{k=1}^{k_\ell}b_k^\ell$ (so that $\alpha_\ell= \operatorname{max}_{k=1,\dots,N} a_k^\ell+\frac1N\sum_{i=1}^k b_i^\ell$), remember that $a_0^\ell=a_N^{\ell-1}$ and we see
\begin{align*}
a_0^\ell +\frac1N \sum_{k=1}^{k_\ell} b_k^{\ell-1}&\leq a_N^{\ell-1}+\frac1N\sum_{k=1}^N b_{k}^{\ell-1}\leq \alpha_{\ell-1} \\ c\sum_{k=1}^{k_\ell}b_k^\ell&\leq cN\alpha_\ell \\ \sum_{k=1}^{k_\ell}d_k^\ell\sqrt{b_k^\ell}&\leq \sqrt{N\sum_{k=1}^{k_\ell}(d_k^\ell)^2}\sqrt{\frac1N \sum_{k=1}^{k_\ell}b_k^\ell}\leq \sqrt{N\sum_{k=1}^N(d_k^\ell)^2}\sqrt{\alpha_\ell},
\end{align*}
so that in total the inequality reads 
$$\alpha_\ell\leq \alpha_{\ell-1}+ cN \alpha_\ell + \sqrt{N\sum_{k=1}^N(d_k^\ell)^2}\sqrt{\alpha_\ell}$$
Now we use the discrete square root Gronwall \Cref{gronwall-sqrt-shiftk} with shifted index for $\alpha_\ell$. This yields
$$\alpha_\ell\leq \left(\sqrt{\alpha_0} + \frac{1}{\sqrt{1-cN}}\sum_{l=1}^\ell\sqrt{N \sum_{k=1}^N (d_k^l ) ^2} \right)^2(1-cN)^{-\ell}$$ which is the desired inequality.
\end{proof}


\subsection{The numeric scheme and the definition of stability}
\label{ssec:scheme}
Let us now define an appropriate notion of stability for a scheme approximating the solution of \eqref{eqn:elastic-eqn}. For this we will perform some heuristical formal a-priori estimates.

Assume formally that $\eta$ is a solution and that $\partial_t \eta$ is an admissible test function.
For the purpose of our formal estimates, assume that it holds
\begin{equation}\label{eqn:econvex-tder}
\langle DE(\eta),\partial_t\eta \rangle = \partial_t E(\eta).
\end{equation}
This is a formal application of the chain rule. 
Then using a test function $\partial_t\eta$ gives
\begin{equation*}
\frac12\partial_t \norm{\partial_t\eta}^2+ \langle DE(\eta), \partial_t\eta \rangle = \langle f,\partial_t\eta\rangle \leq \|f\|_{L^2}\|\partial_t\eta\|_{L^2}.
\end{equation*}
Using this, it follows using a square-root Gronwall type argument that
\begin{equation*}
\frac12\norm{\partial_t\eta(t)}_{L^2}^2 + E(\eta(t))\leq \left( \sqrt{\frac12\norm{\eta_*}_{L^2}^2 +E(\eta_0)}+\frac12\int_0^t\norm{f}_{L^2}\dt\right)^2
\end{equation*}
Accordingly we call an approximation stable if it satisfies an appropriate substitute of the above estimate.

\begin{definition}
Let $\tilde \eta$ be an approximation of the solution.\footnote{At this point we do not specify in which sense it is an approximation, apart from saying that $\tilde\eta$ lies in the correct space, that is \eqref{eta-weaksol-spaces}.} {We say that the approximation is \emph{stable} if it satisfies with some $C\geq 0$ an estimate 
\begin{equation*}
\frac12\norm{\partial_t\tilde\eta(t)}_{L^2}^2 + E(\tilde\eta(t))\leq \left( \sqrt{\frac12\norm{\eta_*}_{L^2}^2 +E(\eta_0)}+C\int_0^t\norm{f}_{L^2}\dt\right)^2.
\end{equation*}}
\end{definition}


Let us now introduce the {\bf Minimizing movement scheme}, that approximates \eqref{eqn:elastic-eqn} with the appropriate stability. The minimizing movement scheme without dissipation is very similar and can be found in~\eqref{def:eta-minimizing-movements-nodiss}.

  Consider the following time-stepping scheme, for the two time scales $0<\tau\leq h$. We follow the scheme of \cite{BenesovaKampschulteSchwarzacher} with the distinction that we keep our scheme discrete in $h$. 
For simplicity and ease of notation\footnote{It is possible to include the cases that $h$ resp. $T$ is not an integer multiple of $\tau$ resp. $h$, and the resulting complications of this are essentially notational.} assume $h=N\tau$ and $T=Mh$ with $N,M\in \N$. Define the discrete times $t_k^\ell:= \ell h+ k\tau$, notice in particular that $t_0^\ell=t_N^{\ell-1}$.
For $k=1,\dots,N$ and $\ell =0,\dots,M$ define the approximation via the step-wise minimization of 
\begin{equation}\label{def:eta-minimizing-movements}
\eta_k^{\ell} = \argmin_{\eta\in\mathcal E}\frac{\tau h}{2}\norm{\frac{\frac{\eta-\eta_{k-1}^\ell}{\tau}-\frac{\eta_k^{\ell-1}-\eta_{k-1}^{\ell-1}}{\tau}}{h}
}_{L^2}^2+E(\eta) +\frac{c\tau^2}{2}\norm{\nabla\frac{\eta_k^\ell-\eta_{k-1}^{\ell}}{\tau}}_{L^2}^2 -\langle f_k^\ell , \eta\rangle_{L^2},
\end{equation}
where we start from the initial conditions $\eta_0^0:=\eta_0$ and for $\ell=0$, the fraction$\frac{\eta_k^{-1}-\eta_{k-1}^{-1}}{\tau}$ is replaced by $\eta_*$. Moreover we take $\eta_0^{\ell+1}:= \eta_N^\ell$, since $t_0^{\ell+1}=t_N^\ell$. The constant $c$ is a factor in front of the regularizer that is chosen large enough to compensate the non-convexity of the energy. The term $f_k^\ell$, a discretization of the right hand side, is defined as 
\begin{equation}\label{eqn:fkl-def}
f_k^\ell:= \fint_0^\tau \fint_0^h f(t_{k-1}^{\ell-1}+s+\sigma) \dx[s] \dx[\sigma].
\end{equation} 
The reason for this particular choice of discretization of $f$ will be apparent later in \Cref{sec:convergence}. The Euler-Lagrange equation of the minimizer indeed approximates the hyperbolic evolution, as can be seen from Lemma~\ref{discrete-solution}. In particular a stabilization term of the form $-c\tau \Delta \partial_t\hat\eta_{(\tau)}^{(h)}$ appears.

\subsection{A priori bounds on energy}

\begin{lemma}
The minimizer $\eta_k^\ell\in\mathcal E$ exists. If $\eta_k^\ell\in \partial\mathcal E$, then a (self-)collision occurred.
\end{lemma} 
\begin{proof}
Existence follows from the direct method. From the lower bound on the determinant \eqref{as:det-lower-bound} we see that we are away from the part of $\partial \mathcal{E}$ corresponding to the case when $\det \nabla\eta$ vanishes somewhere. Therefore $\eta_k^\ell$ can only be in the part of $\partial\mathcal E$ corresponding to the boundary of Ciarlet-Ne\v cas condition, which is exactly the (self-)collision.
\end{proof}
Now denote the piecewise constant approximation
\begin{equation*}
\bar \eta_{(\tau)}^{(h)}(t)=\eta_k^\ell, \quad t\in [t_{k-1}^\ell,t_{k}^\ell)
\end{equation*}
and
\begin{equation*}
\bar f_{(\tau)}^{(h)}(t) =f_k^\ell, \quad t\in [t_{k-1}^\ell,t_{k}^\ell).
\end{equation*}

Denote the piecewise affine approximation
\begin{equation*}
\hat \eta_{(\tau)}^{(h)}(t) = \frac{t-t_{k-1}^\ell}{\tau} \eta_{k-1}^\ell + \frac{t_k^\ell-t}{\tau}\eta_k^\ell \quad\text{for }t\in [t_{k-1}^\ell,t_{k}^\ell),
\end{equation*}
so that $\hat \eta_{(\tau)}^{(h)}(t_k^\ell)=\eta_k^\ell$ and $\hat \eta_{(\tau)}^{(h)}$ is affine on each of the intervals $[t_{k-1}^\ell,t_k^\ell]$.

We now can see that this is a time-discrete solution of our problem, in the following sense

\begin{lemma}\label{discrete-solution} Assume that no collision happened, that is $\eta_k^\ell\in \operatorname{int} \mathcal{E}$ for all $k$ and $\ell$. Then it holds for a.a. times $t\in(0,T)$ that
\begin{equation*}
\frac{\partial_t\hat\eta_{(\tau)}^{(h)}(t)-\partial_t\hat\eta_{(\tau)}^{(h)}(t-h)}{h} +DE\left (\overline\eta_{(\tau)}^{(h)}(t)\right ) -c\tau \Delta \partial_t\hat\eta_{(\tau)}^{(h)}(t) = f(t).
\end{equation*} 
\end{lemma}
\begin{proof}
Since $\eta_k^\ell$ is a minimizer of \eqref{def:eta-minimizing-movements} and it is an interior point, by \eqref{as:gen-E-differentiable} $E$ is differentiable at $\eta_k^\ell$, and we have the Euler-Lagrange equation
\begin{equation*}
\frac{\frac{\eta_k^\ell-\eta_{k-1}^\ell}{\tau}-\frac{\eta_k^{\ell-1}-\eta_{k-1}^{\ell-1}}{\tau}}{h} + DE(\eta_k^\ell) - c\tau \frac{\eta_k^\ell-\eta_{k-1}^\ell}{\tau}= f_k^\ell.  
\end{equation*}
Using the notation above, we have proven the claim.
\end{proof}

\subsubsection*{Approximation of the right hand side}
Now we verify that the discretization of the right hand side is well-behaved.
\begin{lemma}
Let $f\in L^p((0,T);X)$ with $X$ a Banach space and $1\leq p < \infty$. Define
\begin{equation*}
\overline{f}^{(\tau)}(t)=f_k, t\in[t_{k-1},t_k),\quad \text{where  }f_k=\fint_{t_{k-1}}^{t_k}f\dt. 
\end{equation*}
 Then $\norm{\overline{f}^{(\tau)}}_{L^p((0,T);X)}\leq \norm{f}_{L^p((0,T);X)}$ and moreover $\overline{f}^{(\tau)}\to f$ in $L^p((0,T);X)$ as $\tau \to 0$.
\end{lemma}
\begin{proof}
For the first part, use the Jensen inequality
\begin{equation*}
\norm{\overline{f}^{(\tau)}}_{L^p((0,T);X)}^p=\sum_{k=1}^N\tau \norm{\fint_{t_{k-1}}^{t_k} f(t)\dt }_X^p  \leq \sum_{k=1}^N \tau\fint_{t_{k-1}}^{t_k}\norm{f(t)}_X^p \dt =\norm{f}_{L^p((0,T);X)}^p.
\end{equation*}
To prove the convergence, fix $\epsilon>0$ and find $g\in C([0,T];X)$ with $\norm{f-g}_{L^p((0,T);X)}\leq \epsilon$. By uniform continuity we find $\tau>0$ such that $|t-s|<\tau$ implies $\norm{g(t)-g(s)}_X\leq \epsilon$. Then by Jensen inequality and the uniform continuity
\begin{multline*}
\norm{g-\overline g^{(\tau)}}_{L^p([0,T];X)}^p=\sum_{k=1}^N \int_{t_{k-1}}^{t_k} \norm{\fint_{t_{k-1}}^{t_k}g(t)-g(s)\dx[s]}_X^p\dt \\
\leq \sum_{k=1}^N \int_{t_{k-1}}^{t_k} \fint_{t_{k-1}}^{t_k}\norm{g(t)-g(s)}_X^p\dx[s]\dx[t] \leq T\epsilon^p.
\end{multline*}
Moreover by linearity $\overline{f}^{(\tau)}-\overline{g}^{(\tau)} = \overline{(f-g)}^{(\tau)}$, so $
\norm{\overline{f}^{(\tau)}-\overline{g}^{(\tau)}}_{L^p((0,T);X)}^p\leq \norm{f-g}_{L^p((0,T);X)}^p$ by the first part. We conclude the proof by the triangle inequality
\begin{multline*}
\norm{f-\overline{f}^{(\tau)}}_{L^p((0,T);X)} \leq \norm{f-g}_{L^p((0,T);X)}+\norm{g-\overline{g}^{(\tau)}}_{L^p((0,T);X)}+\norm{\overline{g}^{(\tau)}-\overline{f}^{(\tau)}}_{L^p((0,T);X)}\\
 \leq \epsilon+\sqrt[p]{T} \epsilon +\epsilon.
\end{multline*}

\end{proof}

\begin{lemma}\label{lem:approx-rhs-convergence}
Let $f\in L^p((0,T);X)$ with $X$ a Banach space and $1\leq p<\infty$. Extend $f$ by $0$ outside $(0,T)$ and define
\begin{equation*}
f_k^\ell:= \fint_0^\tau \fint_0^h f(t_{k-1}^{\ell-1}+s+\sigma) \dx[s] \dx[\sigma], \quad \bar f^{(h)}_{(\tau)}(t):= f_k^\ell, \quad t\in [t_{k-1}^\ell,t_k^\ell).
\end{equation*}
Then $\norm{\bar f_{(\tau)}^{(h)}}_{L^p((0,T);X)}\leq\norm{ f}_{L^p((0,T);X)}$ and moreover $\bar f_{(\tau)}^{(h)} \to f$ in $L^p((0,T);X)$ as $h,\tau\to 0$.
\end{lemma}
\begin{proof}
For the first claim, use twice Jensen inequality as follows
\begin{multline*}
\norm{\bar f_{(\tau)}^{(h)}}_{L^p((0,T);X)}^p = \sum_{\ell=1}^M\sum_{k=1}^N \tau\norm{f_k^\ell}_X^p= \sum_{\ell=1}^M\sum_{k=1}^N \tau\norm{\fint_0^\tau \fint_0^h f(t_{k-1}^{\ell-1}+s+\sigma) \dx[s] \dx[\sigma]}_X^p \\
\leq \sum_{\ell=1}^M\sum_{k=1}^N  \int_0^\tau \norm{\fint_0^h f(t_{k-1}^{\ell-1}+s+\sigma) \dx[s] }_X^p \dx[\sigma]=\sum_{\ell=1}^M\sum_{k=1}^N  \int_{t_{k-1}^{\ell-1}}^{t_k^{\ell-1}} \norm{\fint_0^h f(t+s) \dx[s] }_X^p\dx[t] \\
= \int_0^T \norm{\fint_0^h f(t+s) \dx[s] }_X^p\dx[t]
 \leq \int_0^T \fint_0^h \norm{f(t+s)}_X^p \dx[s] \dx[t] = \fint_0^h \int_0^T \norm{f(t+s)}_X^p\dx[t]  \dx[s] \\
 \leq \fint_0^h \norm{f}_{L^p((0,T);X)}^p \dx[s] = \norm{f}_{L^p((0,T);X)}^p.
\end{multline*}

Fix $\epsilon>0$. Find $g\in C([0,T];X)$ with $\norm{f-g}_{L^p((0,T);X)}\leq \epsilon$. Then, by uniform continuity of $g$, find $h_0>0$ such that for all $|t-s|
\leq h_0$ it holds $\norm{g(t)-g(s)}_X\leq \epsilon$. Then, using this and Jensen inequality, for $\tau\leq h\leq h_0$
\begin{align*}
\norm{g-\overline{g}_{(\tau)}^{(h)}}_{L^p((0,T);X)}^p= \sum_{\ell=1}^M\sum_{k=1}^N\int_{t_{k-1}^\ell}^{t_k^\ell} \norm{\fint_0^\tau \fint_0^h g(t)-g(t_{k-1}^{\ell-1}+s+\sigma) \dx[s]\dx[\sigma]}_X^p \dt\\
\leq \sum_{\ell=1}^M\sum_{k=1}^N\int_{t_{k-1}^\ell}^{t_k^\ell} \fint_0^\tau \fint_0^h \norm{g(t)-g(t_{k-1}^{\ell-1}+s+\sigma)}_X^p \dx[s]\dx[\sigma]\dt\leq T \epsilon^p.
\end{align*}
Further, by linearity $\overline f_{(\tau)}^{(h)}-\overline g_{(\tau)}^{(h)} = (\overline{f-g})_{(\tau)}^{(h)} $ so by the first part $\norm{\overline f_{(\tau)}^{(h)}-\overline g_{(\tau)}^{(h)}}_{L^p((0,T);X)}^p \leq  \norm{f-g}_{L^p((0,T);X)}^p$. We then conclude the proof by the triangle inequality
\begin{multline*}
\norm{\overline{f}_{(\tau)}^{(h)}-f}_{L^p((0,T);X)}\leq \norm{\overline{f}_{(\tau)}^{(h)}-\overline{g}_{(\tau)}^{(h)}}_{L^p((0,T);X)}+\norm{\overline{g}_{(\tau)}^{(h)}-g}_{L^p((0,T);X)}+\norm{g-f}_{L^p((0,T);X)}\\
\leq \epsilon+ \sqrt[p]{T} \epsilon+\epsilon.
\end{multline*}
\end{proof}

\subsection{Stability -- elastic solid}
\label{ssec:stab}

We can now state our main stability result for the approximation \eqref{def:eta-minimizing-movements}. Note that we have an energy estimate with bound given by the initial condition and the right hand side, and in particular the energy can not grow in time in case no right hand side is considered.

\begin{theorem}[Stability with dissipation]\label{thm:stability-solid-stabilized}
There exists a $h_0>0$ and $c>0$ depending on $E(\eta_0)$, $\norm{\eta_*}_{L^2(Q)}$, $\norm{f}_{L^2((0,T)\times Q)}$, the assumptions on $E$ and $T$, such that for all $N\tau=h\leq h_0$ if the corresponding approximation $\eta_k^\ell$ does not reach a collision, i.e. satisfies $\eta_k^\ell \in \operatorname{int}\mathcal E$ for all $k$ and $\ell$, then the following stability estimate holds
\begin{equation}\label{eqn:stability-diss}
\max_{k=1,\dots,N} \left(E(\eta_k^\ell) + \frac1N\sum_{i=1}^k \norm{\frac{\eta_k^\ell-\eta_{k-1}^\ell}{\tau}}_{L^2}^2\right)
 \leq \left(\sqrt{E(\eta_0)+\frac12\norm{\eta_*}_{L^2}^2}+\norm{f}_{L^2((0,h\ell)\times Q)} \right)^2.
\end{equation}
\end{theorem}

\begin{proof}
To ease the notation, the norm $\norm{\cdot}$ without any index is the $L^2(Q)$ norm, and $\langle \cdot,\cdot\rangle$ is the $L^2(Q)$ scalar product (or dual pairing of $X$ and $X^*$ in the $DE$ terms).

We proceed by induction on $\ell$. Thus assume that the inequality \eqref{eqn:stability-diss} holds for $\ell-1$ (and every $k$), and we want to prove it for $\ell$.

We first show the following auxiliary estimate for $k=1,\dots,N$:
\begin{equation}\label{eqn:para-estimate}
E(\eta_{k}^\ell)\leq K:= \left(\sqrt{E(\eta_0)+\frac12\norm{\eta_*}_{L^2}^2}+\norm{f}_{L^2((0, T)\times Q)} \right)^2 + h_0\norm{f}_{L^2((0,T)\times Q)}^2.  
\end{equation}
For this, take $\eta_{k-1}^\ell$ as competitor in \eqref{def:eta-minimizing-movements}, which implies

\begin{align*}
\frac{\tau h}{2}\norm{\frac{\frac{\eta_k^\ell-\eta_{k-1}^\ell}{\tau}-\frac{\eta_k^{\ell-1}-\eta_{k-1}^{\ell-1}}{\tau}}{h}
}_{L^2}^2+E(\eta_k^\ell) +\frac{c\tau^2}{2}\norm{\nabla\frac{\eta_k^\ell-\eta_{k-1}^{\ell}}{\tau}}_{L^2}^2 
\\
\leq \frac{\tau h}{2}\norm{\frac{\eta_k^{\ell-1}-\eta_{k-1}^{\ell-1}}{\tau h}}^2+E(\eta_{k-1}^\ell) +\langle f_k^\ell , (\eta_{k}^\ell-\eta_{k-1}^\ell)\rangle_{L^2}.
\end{align*}
We then find that
\begin{align*}
&\frac{\tau}{2h}\norm{\frac{\eta_k^\ell-\eta_{k-1}^\ell}{\tau}-\frac{\eta_{k}^{\ell-1}-\eta_{k-1}^{\ell-1}}{\tau}
}_{L^2}^2+E(\eta_k^{\ell}) +\frac{c\tau^2}{2}\norm{\nabla\frac{\eta_k^\ell-\eta_{k-1}^\ell}{\tau}}_{L^2}^2 \\
&\leq \frac{\tau}{2h}\norm{\frac{\eta_{k}^{\ell-1}-\eta_{k-1}^{\ell-1}}{\tau}}^2+E(\eta_{k-1}^\ell) +\langle f_k^\ell , (\eta_{k}-\eta_{k-1})\rangle_{L^2}
\\
&\leq \frac{\tau}{2h}\norm{\frac{\eta_{k}^{\ell-1}-\eta_{k-1}^{\ell-1}}{\tau}}^2+E(\eta_{k-1}^\ell) + \norm{f_k^\ell}\tau\norm{\frac{\eta_{k}^{\ell-1}-\eta_{k-1}^{\ell-1}}{\tau}}+\norm{f_k^\ell}\tau\norm{\frac{\eta_{k}^\ell-\eta_{k-1}^\ell}{\tau}-\frac{\eta_k^{\ell-1}-\eta_{k-1}^{\ell-1}}{\tau}}
\\
&\leq \frac{\tau}{h}\norm{\frac{\eta_{k}^{\ell-1}-\eta_{k-1}^{\ell-1}}{\tau}}^2+E(\eta_{k-1}^\ell) +h\tau \norm{f_k^\ell}^2 + \frac{\tau}{2h}\norm{\frac{\eta_k^\ell-\eta_{k-1}^\ell}{\tau}-\frac{\eta_k^{\ell-1}-\eta_{k-1}^{\ell-1}}{\tau}}^2.
\end{align*}
Subtracting $\frac{\tau}{2h}\norm{\frac{\eta_k^\ell-\eta_{k-1}^\ell}{\tau} -\frac{\eta_k^{\ell-1}-\eta_{k-1}^{\ell-1}}{\tau}}^2$ on both sides and summing over $k$, we find that the $E(\eta_k^\ell)$ term telescopes, so that 
\begin{align*}&E(\eta_k^\ell) +\sum_{i=1}^k \frac{c\tau^2}{2}\norm{\nabla\frac{\eta_i^\ell-\eta_{i-1}^\ell}{\tau}}^2 
\leq E(\eta_0^\ell) + \sum_{i=1}^k \frac{\tau}{h} \norm{\frac{\eta_i^{\ell-1}-\eta_{i-1}^{\ell-1}}{\tau}}^2 + h\sum_{i=1}^k \tau \norm{f_i^\ell}^2.
\end{align*}
Now remember that we have $\eta_0^\ell=\eta_N^{\ell-1}$ and \eqref{eqn:stability-diss} with $\ell-1$ holds, so after a Jensen inequality $h\sum_{i=1}^k \tau \norm{f_i^\ell}^2 \leq \norm{f}_{L^2((\ell-1)h,\ell h)\times Q)}^2$ and $0<h\leq h_0$ we show \eqref{eqn:para-estimate}.

	
Equipped with this, we come to the main estimate.
Since we assume that $\eta_k^\ell\in \operatorname{int} \mathcal{E}$, $E$ is differentiable at $\eta_k^\ell$ and we can test the minimizer with $\frac{\eta_k^\ell-\eta_{k-1}^\ell}{\tau}$, which gives
\begin{equation*}
\left\langle \frac{\frac{\eta_k^\ell-\eta_{k-1}^\ell}{\tau}-\frac{\eta_k^{\ell-1}-\eta_{k-1}^{\ell-1}}{\tau}}{h}, \frac{\eta_k^\ell-\eta_{k-1}^\ell}{\tau} \right\rangle+\left\langle DE(\eta_k^\ell),\frac{\eta_k^\ell-\eta_{k-1}^\ell}{\tau}\right\rangle +c\tau \!\norm{\nabla\frac{\eta_k^\ell-\eta_{k-1}^\ell}{\tau}}^2\!\!\!=\!\left\langle f_k^\ell,\frac{\eta_k^\ell-\eta_{k-1}^\ell}{\tau}\right\rangle\!.
\end{equation*}
Using $a(a-b)=\frac{a^2}{2}-\frac{b^2}{2}+\frac{(a-b)^2}{2}$, we obtain for the first term
\begin{multline*}
\left\langle \frac{\frac{\eta_k^\ell-\eta_{k-1}^\ell}{\tau}-\frac{\eta_k^{\ell-1}-\eta_{k-1}^{\ell-1}}{\tau}}{h}, \frac{\eta_k^\ell-\eta_{k-1}^\ell}{\tau} \right\rangle\\
=  \frac{1}{2h}\norm{\frac{\eta_k^\ell-\eta_{k-1}^\ell}{\tau}}^2  + \frac{1}{2h}\norm{\frac{\eta_k^\ell-\eta_{k-1}^\ell}{\tau} - \frac{\eta_k^{\ell-1}-\eta_{k-1}^{\ell-1}}{\tau}}^2 -\frac{1}{2h}\norm{\frac{\eta_k^{\ell-1}-\eta_{k-1}^{\ell-1}}{\tau}}^2.
\end{multline*}

Now multiply by $\tau$, omit the middle term, use \Cref{thm:E-noncovexity} (note carefully that by \eqref{eqn:para-estimate} the resulting constant $C_1$ is independent of $\ell$ and $k$), and obtain
\begin{multline*}
\frac{\tau}{2h} \norm{\frac{\eta_k^\ell-\eta_{k-1}^\ell}{\tau}}^2
\!+E(\eta_k^\ell)+c\tau^2 \norm{\nabla\frac{\eta_k^\ell-\eta_{k-1}^\ell}{\tau}}^2 \\ \leq \frac{\tau}{2h}\norm{\frac{\eta_k^{\ell-1}-\eta_k^{\ell-1}}{\tau}}^2 \!+ E(\eta_{k-1}^\ell)+\tau\!\left\langle f_k^\ell,\frac{\eta_k^\ell-\eta_{k-1}^\ell}{\tau}\right\rangle +C_1\tau^2\norm{\nabla\frac{\eta_k^\ell-\eta_{k-1}^\ell}{\tau}}^2,
\end{multline*}
where we precisely here choose $c>2C_1$.
Hence 
\begin{align}
\label{eq:est}
\begin{aligned}
&\frac{\tau}{2h} \norm{\frac{\eta_k^\ell-\eta_{k-1}^\ell}{\tau}}^2
\!+E(\eta_k^\ell) +C_1\tau^2 \norm{\nabla\frac{\eta_k^\ell-\eta_{k-1}^\ell}{\tau}}^2
\\
&\quad \leq \frac{\tau}{2h}\norm{\frac{\eta_k^{\ell-1}-\eta_k^{\ell-1}}{\tau}}^2 \!+ E(\eta_{k-1}^\ell)+\tau\!\left\langle f_k^\ell,\frac{\eta_k^\ell-\eta_{k-1}^\ell}{\tau}\right\rangle.
\end{aligned}
\end{align}
Using the inequality
\begin{equation*}
\tau\left\langle f_k^\ell, \frac{\eta_k^\ell-\eta_{k-1}^\ell}{\tau}\right\rangle \leq \tau \norm{f_k^\ell} \norm{\frac{\eta_k^\ell-\eta_{k-1}^\ell}{\tau}}
\end{equation*}
we get, denoting $a_k^\ell = E(\eta_k^\ell)$, $b_k^\ell =\frac12\norm{\frac{\eta_k^\ell-\eta_{k-1}^\ell}{\tau}}^2$, $d_k^\ell=\tau \norm{f_k^\ell}$ and $N=h/\tau$ the inequality
\begin{equation*}
a_k^\ell+\frac1N b_k^\ell\leq a_{k-1}^\ell + \frac1N b_k^{\ell-1} + d_k^\ell \sqrt{b_k^\ell}.
\end{equation*}
Now we are in a position to use the Two-scale Gronwall inequality with square root \Cref{twoscale-gronwall-sqrt}. Thus, we obtain
$$
\max_{k=1,\dots,N} \left( a_k^\ell + \frac1N\sum_{i=1}^k b_i^\ell\right) \leq \left(\sqrt{a_0^0+b_0^0}+\sum_{l=1}^\ell\sqrt{N \sum_{k=1}^N (d_k^\ell)^2}\right)^2, \quad \ell=1,\dots,M-1.
$$
So this reads
$$
\max_{k=1,\dots,N} \left(E(\eta_k^\ell) + \frac1N\sum_{i=1}^k \norm{\frac{\eta_k^\ell-\eta_{k-1}^\ell}{\tau}}^2\right) \leq \left(\sqrt{E(\eta_0)+\frac12\norm{\eta_*}^2}+\sum_{l=1}^\ell h \sqrt{\frac1N \sum_{k=1}^N \|f_k^\ell\|^2}\right)^2.
$$
Finally, using \Cref{lem:approx-rhs-convergence}, from which we get by Jensen inequality 
\begin{equation*}
\sum_{l=1}^\ell h \sqrt{\frac1N \sum_{k=1}^N \|f_k^\ell\|^2}\leq \norm{f}_{L^2((0,h\ell)\times Q)},
\end{equation*}
and the induction on $\ell$ is finished.
\end{proof}

\begin{lemma}[Dissipation estimate]\label{dissipation-estimate} 
Further, the approximation satisfies 
\begin{equation}
c\tau \norm{\partial_t\nabla \hat\eta_{(\tau)}^{(h)}}_{L^2((0,T)\times Q)}^2 =c \tau \sum_{\ell=1}^M\sum_{k=1}^N \tau \norm{\nabla\frac{\eta_k^\ell-\eta_{k-1}^\ell}{\tau}}^2 \leq C
\end{equation}
with $C$ depending on $\eta_0$, $\eta_*$, $f$. 
\end{lemma}
\begin{proof}
When we sum \eqref{eq:est}
over $k$ and $\ell$, we obtain 
$$C_1\tau \sum_{\ell=1}^M \sum_{k=1}^N \tau \norm{\nabla\frac{\eta_k^\ell-\eta_{k-1}^\ell}{\tau}}^2 \leq \frac12 \|\eta_*\|^2 +E(0)-E_{\min}+\sum_{\ell=1}^M \sum_{k=1}^N\tau \left\langle f_k^\ell,\frac{\eta_k^\ell-\eta_{k-1}^\ell}{\tau}\right\rangle 
$$
Estimate
$$\sum_{\ell=1}^M \sum_{k=1}^N\tau \left\langle f_k^\ell,\frac{\eta_k^\ell-\eta_{k-1}^\ell}{\tau}\right\rangle  \leq \sum_{\ell=1}^M h \sqrt{\frac1N \sum_{k=1}^N \norm{f_k^\ell}^2 } \sqrt{\frac1N \sum_{k=1}^N \norm{\frac{\eta_k^\ell-\eta_{k-1}^\ell}{\tau}}^2}$$
Since by the stability estimate \ref{thm:stability-solid-stabilized} the last term is bounded by a constant, we are finished.
\end{proof}

\subsection{Stability -- general}

For the general case, namely assuming Assumption~\ref{assumptions-general}, the minimizing movement approximation \eqref{def:eta-minimizing-movements} becomes

\begin{equation}\label{def:eta-minimizing-movements-gen}
\eta_k^{\ell} = \argmin_{\eta\in\mathcal E}\frac{\tau h}{2}\norm{\frac{\frac{\eta-\eta_{k-1}^\ell}{\tau}-\frac{\eta_k^{\ell-1}-\eta_{k-1}^{\ell-1}}{\tau}}{h}
}_{H}^2+E(\eta) +\frac{c\tau^2}{2}\left\langle L\frac{\eta_k^\ell-\eta_k^{\ell-1}}{\tau},\frac{\eta_k^\ell-\eta_k^{\ell-1}}{\tau}\right\rangle_Z -\langle f_k^\ell , \eta\rangle_H,
\end{equation}

It is readily seen that the argument above for the elastic solids goes through in the general case. We thus have 

\begin{theorem}
There exists a $h_0>0$ and $c_0>0$ depending on $E(\eta_0)$, $\norm{\eta_*}_H$, $\norm{f}_{L^2((0,T);H)}$, the assumptions on $E$ and $T$ such that for all $N\tau=h\leq h_0$ if $c>c_0$ in \eqref{def:eta-minimizing-movements-gen} and the corresponding approximation  $\eta_k^\ell$ satisfies $\eta_k^\ell \in \operatorname{int} \mathcal E$ we have the stability estimate
\begin{equation*}
\max_{k=1,\dots,N} \left(E(\eta_k^\ell) +\frac{1}{2N} \sum_{i=1}^k\norm{\frac{\eta_i^\ell-\eta_{i-1}^\ell}{\tau}}_H^2 \right) \leq  \left(\sqrt{E(\eta_0)+\frac12 \norm{\eta_*}_H^2}+ \norm{f}_{L^2((0,\ell h);H)}\right)^2,
\end{equation*}
further the approximation satisfies 
\begin{equation}
c\lambda \tau \sum_{\ell=1}^M\sum_{k=1}^N \tau \norm{\frac{\eta_k^\ell-\eta_{k-1}^\ell}{\tau}}_Z^2 \leq C
\end{equation}
with $C$ depending on $\eta_0$, $\eta_*$, $f$. 
\end{theorem}

\subsection{Stability -- ODE} 
\label{ssec:stabend}
It is worth mentioning that second-order ordinary differential equations fit into our framework, and therefore our approximation is also stable here.

 Namely, consider the equation with unknown $x\colon (0,T)\to \R^n$
\begin{align}
\begin{aligned}
x''+ \nabla E(x) &= f,\\
x(t)&\in\mathcal E,\\
x(0)&=x_0,\\
x'(0)&=x_*,
\end{aligned}
\end{align}
where $f\in L^2((0,T);\R^n)$, $x_0,x_*\in\R^n$, $E\colon \R^{n} \to (-\infty,\infty]$ and $\mathcal E \subset \R^n$ closed with nonempty interior (and $X=H=\R^n$). 

Moreover let $E\in C(\R^n,(-\infty,\infty])\cap C^2(\mathcal{E})$ and $\mathcal E \subset \{z\in\R^n: E(z)<\infty\}$ and $E$ be coercive in the sense $\lim_{|x|\to\infty}E(x)=\infty$. We only need to show the non-convexity estimate \eqref{as:nonconvexity-general-Z}, the rest of the properties are clear.

\begin{lemma}
The function $E$ satisfies the non-convexity estimate: For all $x,y\in \mathcal{E}$ we find 
\begin{equation*}
 \nabla E(y) \cdot(y-x) \geq E(y)-E(x) - C|x-y|^2,
\end{equation*}
where $C$ depends on $K=\max(E(x),E(y))$.
\end{lemma}
\begin{proof}
Let $x,y\in\mathcal{E}$. Let us introduce a cutoff $\tilde E$ of $E$ such that $\tilde E\colon \R^n\to \R$, $\tilde E\in C^2(\R^n)$, $\tilde E(z)=E(z)$ for all $z\in \mathcal E$ with $E(z)\leq K$, and $\|\tilde E\|_{C^2(\R^n)}$ norm depends only on $K$. This we can certainly achieve by smoothly cutting off $E$ near the compact set $\{z\in\mathcal E : E(z)\leq K \}$, which depends only on $K$.

Then we find that for some $z$ on the line segment between $x$ and $y$
\begin{equation*}
\tilde E(x)=\tilde E(y)+\nabla \tilde E(y)\cdot(x-y) +\frac12 \nabla^2 \tilde E(z) : (y-x)\otimes (y-x)
\end{equation*}
Since $E=\tilde E$ on a neighbourhood of $x$ and $y$, we find that
\begin{equation*}
\nabla E(y)\cdot (y-x) \geq E(y)-E(x) - \frac12\|\tilde E\|_{C^2(\R^n)}|y-x|^2.
\end{equation*}
which finishes the proof.
\end{proof}

\subsection{Stability estimate without stabilizing term--elastic solids}

In the approximation \eqref{def:eta-minimizing-movements} we have used a stabilizing dissipation term. In this section, we show that if we assume the stronger convexity assumption \eqref{as:E2-unif-convex}, we obtain a stability estimate without introducing the stabilizing dissipation term. Although this estimate is weaker than the previous one, we consider it possibly of independent interest due to easier implementation of this scheme.
\\
Therefore, throughout this section, we assume \eqref{as:E2-unif-convex} instead of \eqref{as:E2-convex}, respectively \eqref{as:nonconvexity-general-H} instead of \eqref{as:nonconvexity-general-Z}.

The time step \eqref{def:eta-minimizing-movements} is now replaced by the following minimization:
\begin{equation}\label{def:eta-minimizing-movements-nodiss}
\eta_k^{\ell} = \argmin_{\eta\in\mathcal E}\frac{\tau h}{2}\norm{\frac{\frac{\eta-\eta_{k-1}^\ell}{\tau}-\frac{\eta_k^{\ell-1}-\eta_{k-1}^{\ell-1}}{\tau}}{h}
}_{L^2}^2+E(\eta)-\langle f_k^\ell , \eta\rangle,
\end{equation}
where as before $\eta_0^0:=\eta_0$ and for $\ell=0$, the fraction$\frac{\eta_k^{-1}-\eta_{k-1}^{-1}}{\tau}$ is replaced by $\eta_*$. Moreover we take $\eta_0^{\ell+1}:= \eta_N^\ell$, since $t_0^{\ell+1}=t_N^\ell$. The term $f_k^\ell$ is as before defined by \eqref{eqn:fkl-def}.

Analogously as before we find (denoting the piecewise constant/affine interpolations) that an approximation can be constructed.
\begin{lemma}
The minimizer $\eta_k^\ell\in\mathcal E$ exists. If $\eta_k^\ell\in \partial\mathcal E$, then a (self-)collision occurred. Assume that no collision happened, that is $\eta_k^\ell\in \operatorname{int} \mathcal{E}$ for all $k$ and $\ell$. Then it holds for a.a. times $t\in(0,T)$ that 
\begin{equation*}
\frac{\partial_t\hat\eta_{(\tau)}^{(h)}(t)-\partial_t\hat\eta_{(\tau)}^{(h)}(t-h)}{h}  +DE\left (\overline\eta_{(\tau)}^{(h)}(t)\right ) = f(t).
\end{equation*}
\end{lemma} 
Now we present the proof of the stability in this case. Notice in particular on the right hand side the term which linearly depends on $\tau$ and approaches $1$ with $\tau\to 0$.

\begin{theorem}[Stability for elastic solid]\label{thm:stability-solid}
There exists a $h_0>0$ and $C>0$ depending on $E(\eta_0)$, $\norm{\eta_*}_{L^2(Q)}$, $\norm{f}_{L^2((0,T)\times Q)}$, the assumptions on $E$ and $T$, such that for all $N\tau=h\leq h_0$ with $h\ell\leq T$ the following holds: If the corresponding approximation $\eta_k^\ell$ does not reach a collision, i.e. it satisfies $\eta_k^\ell \in \operatorname{int}\mathcal E$ for all $k$ and $\ell$, then the following stability estimate holds
\begin{align}
\label{eq:stability}
\max_{k=1,\dots,N}\! \left(E(\eta_k^\ell) + \frac1N\sum_{i=1}^k \norm{\frac{\eta_k^\ell-\eta_{k-1}^\ell}{\tau}}_{L^2}^2\right)\! \leq\! \left(\sqrt{E(\eta_0)+\frac12\norm{\eta_*}_{L^2}^2}+\norm{f}_{L^2((0,h\ell)\times Q)} \right)^{\!\!2}\!\! (1+4C\tau h \ell).
\end{align}
\end{theorem}

\begin{proof}
As before to ease the notation, the norm $\norm{\cdot}$ without any index is the $L^2(Q)$ norm, and $\langle \cdot,\cdot\rangle$ is the $L^2(Q)$ scalar product (or dual pairing of $X$ and $X^*$ in the $DE$ terms).

We proceed by induction on $\ell$. Thus assume that the inequality \eqref{eq:stability} holds for $\ell-1$ (and every $k$), and we want to prove it for $\ell$. Analogous to the damped case we need the following auxiliary estimate for $k=1,\dots,N$:
\begin{equation}\label{eqn:hyper-estimate}
E(\eta_{k}^\ell)\leq K:= (1+4C\tau T)\left(\sqrt{E(\eta_0)+\frac12\norm{\eta_*}_{L^2}^2}+\norm{f}_{L^2((0, T)\times Q)} \right)^2 + h_0\norm{f}_{L^2((0,T)\times Q)}^2.  
\end{equation}
But this estimate does however follow line by line by the argument in the proof of Theorem~\ref{thm:stability-solid-stabilized}. Indeed, the damping part is not used for any of the absorbed terms.

Hence we assume that $\eta_k^\ell\in \operatorname{int} \mathcal{E}$, with  $E$ is differentiable at $\eta_k^\ell$ and we can test the minimizer with a uniform bound. Now taking $\frac{\eta_k^\ell-\eta_{k-1}^\ell}{\tau}$ as a test function gives
\begin{equation*}
\left\langle \frac{\frac{\eta_k^\ell-\eta_{k-1}^\ell}{\tau}-\frac{\eta_k^{\ell-1}-\eta_{k-1}^{\ell-1}}{\tau}}{h}, \frac{\eta_k^\ell-\eta_{k-1}^\ell}{\tau} \right\rangle+\left\langle DE(\eta_k^\ell),\frac{\eta_k^\ell-\eta_{k-1}^\ell}{\tau}\right\rangle=\left\langle f_k^\ell,\frac{\eta_k^\ell-\eta_{k-1}^\ell}{\tau}\right\rangle.
\end{equation*}
Using $a(a-b)=\frac{a^2}{2}-\frac{b^2}{2}+\frac{(a-b)^2}{2}$, we obtain for the first term
\begin{multline*}
\left\langle \frac{\frac{\eta_k^\ell-\eta_{k-1}^\ell}{\tau}-\frac{\eta_k^{\ell-1}-\eta_{k-1}^{\ell-1}}{\tau}}{h}, \frac{\eta_k^\ell-\eta_{k-1}^\ell}{\tau} \right\rangle\\
=  \frac{1}{2h}\norm{\frac{\eta_k^\ell-\eta_{k-1}^\ell}{\tau}}^2  + \frac{1}{2h}\norm{\frac{\eta_k^\ell-\eta_{k-1}^\ell}{\tau} - \frac{\eta_k^{\ell-1}-\eta_{k-1}^{\ell-1}}{\tau}}^2 -\frac{1}{2h}\norm{\frac{\eta_k^{\ell-1}-\eta_{k-1}^{\ell-1}}{\tau}}^2.
\end{multline*}

Now multiply by $\tau$, omit the middle term, use the non convexity estimate from \Cref{thm:E-noncovexity} and obtain
\begin{equation*}
\frac{\tau}{2h} \norm{\frac{\eta_k^\ell-\eta_{k-1}^\ell}{\tau}}^2
\!+E(\eta_k^\ell) \leq \frac{\tau}{2h}\norm{\frac{\eta_k^{\ell-1}-\eta_k^{\ell-1}}{\tau}}^2 \!+ E(\eta_{k-1}^\ell)+\tau\!\left\langle f_k^\ell,\frac{\eta_k^\ell-\eta_{k-1}^\ell}{\tau}\right\rangle + C\tau^2\norm{\frac{\eta_k^\ell-\eta_{k-1}^\ell}{\tau}}^2
\end{equation*}
where the constant $C$ depends only on $K$.

Using the inequality
\begin{equation*}
\tau\left\langle f_k^\ell, \frac{\eta_k^\ell-\eta_{k-1}^\ell}{\tau}\right\rangle \leq \tau \norm{f_k^\ell} \norm{\frac{\eta_k^\ell-\eta_{k-1}^\ell}{\tau}}
\end{equation*}
we get, denoting $a_k^\ell = E(\eta_k^\ell)$, $b_k^\ell =\frac12\norm{\frac{\eta_k^\ell-\eta_{k-1}^\ell}{\tau}}^2$, $c=2C\tau^2$, $d_k^\ell=\tau \norm{f_k^\ell}$ and $N=h/\tau$ the inequality
\begin{equation*}
a_k^\ell+\frac1N b_k^\ell\leq a_{k-1}^\ell + \frac1N b_k^{\ell-1} + c b_k^\ell + d_k^\ell \sqrt{b_k^\ell}.
\end{equation*}
Now we are in a position to use the Two-scale Gronwall inequality with square root \Cref{twoscale-gronwall-sqrt}. Thus, provided $c<\frac1N$ (i.e. $\tau h <\frac{1}{2C}$, which is guaranteed by $h<h_0:=\sqrt{2C}$), we obtain
$$
\max_{k=1,\dots,N} \left( a_k^\ell + \frac1N\sum_{i=1}^k b_i^\ell\right) \leq \left(\sqrt{a_0^0+b_0^0}+\frac{1}{\sqrt{1-cN}}\sum_{l=1}^\ell\sqrt{N \sum_{k=1}^N (d_k^\ell)^2}\right)^2 (1-cN)^{-\ell},\quad \ell=1,\dots,M-1.
$$
Further, for $c\leq\frac{1}{2N}$ (i.e. $\tau h \leq\frac{1}{4C}$, which is guaranteed by $h\leq h_0:=\frac{1}{2\sqrt C}$) we have $(1-cN)^{-\ell}\leq 1+2cN\ell$, so that this implies
$$
\max_{k=1,\dots,N} \left(E(\eta_k^\ell) + \frac1N\sum_{i=1}^k \norm{\frac{\eta_k^\ell-\eta_{k-1}^\ell}{\tau}}^2\right) \leq \left(\sqrt{E(\eta_0)+\frac12\norm{\eta_*}^2}+\sum_{l=1}^\ell h \sqrt{\frac1N \sum_{k=1}^N \|f_k^\ell\|^2}\right)^2 ( 1+4C\tau h \ell).
$$

Finally using \Cref{lem:approx-rhs-convergence}, from which we know, using a Jensen inequality
\begin{equation*}
\sum_{l=1}^\ell h \sqrt{\frac1N \sum_{k=1}^N \|f_k^\ell\|^2} \leq 
\sqrt{\sum_{l=1}^\ell  \sum_{k=1}^N \tau \|f_k^\ell\|^2}  \leq \norm{f}_{L^2((0,h\ell)\times Q)}
\end{equation*}
the proof is finished.

\end{proof}

\subsection{Stability estimate without stabilzing term--general}
The conditions in Assumption~\ref{assumptions-general}, replacing $(A.5)$ by $(A.5')$ are chosen precisely in such a way that we have the same stability result. Hence the proof is the same as the proof for elastic solids after making the obvious changes, therefore we omit it. And we reach with the following estimate:
\begin{theorem}
There exists a $h_0>0$ and $C>0$ depending on $E(\eta_0)$, $\norm{\eta_*}_H$, $\norm{f}_{L^2((0,T);H)}$, the assumptions on $E$ and $T$ such that for all $N\tau=h\leq h_0$ if the corresponding approximation $\eta_k^\ell$ satisfies $\eta_k^\ell \in \operatorname{int} \mathcal E$ we have the stability estimate
\begin{equation*}
\max_{k=1,\dots,N} \left(E(\eta_k^\ell) +\frac{1}{2N} \sum_{i=1}^k\norm{\frac{\eta_i^\ell-\eta_{i-1}^\ell}{\tau}}_H^2 \right) \leq  \left(\sqrt{E(\eta_0)+\frac12 \norm{\eta_*}_H^2}+ \frac12\norm{f}_{L^2((0,T);H)}\right)^2(1+4C\tau h\ell),
\end{equation*}
for all $\ell$, with $\ell h\leq T$.
\end{theorem}

\begin{remark} Another possibility, in case \eqref{as:E2-unif-convex} holds, we can achieve stabilization by using $c\tau \partial_t\hat\eta_{(\tau)}^{(h)}$, that is the minimization
$$\eta_k^{\ell} = \argmin_{\eta\in\mathcal E}\frac{\tau h}{2}\norm{\frac{\frac{\eta-\eta_{k-1}^\ell}{\tau}-\frac{\eta_k^{\ell-1}-\eta_{k-1}^{\ell-1}}{\tau}}{h}
}_{L^2}^2+E(\eta) +\frac{c\tau^2}{2}\norm{\frac{\eta_k^\ell-\eta_k^{\ell-1}}{\tau}}^2 -\langle f_k^\ell , \eta\rangle.
$$
The proof of the stability goes through the same way as in \Cref{thm:stability-solid-stabilized}, utilizing the estimate of \Cref{thm:E-noncovexityii} instead of \Cref{thm:E-noncovexity}, and gives same estimate as in \Cref{thm:stability-solid-stabilized}.
\end{remark}

\section{Convergence rate}\label{sec:convergence}
Now our focus will be on quantifying the convergence results, namely showing that the under some regularity conditions, our scheme \emph{converges to the solution with a linear rate}. Here we focus on the model case of elastic  energies. This means we stick to \Cref{assumptions-energy-solid}.
Further, throughout the entire section we assume the $W^{k_0,2}$-case, that is
\begin{equation}\label{eqn:Wk02-setting}
X=W^{k_0,2}(Q), \quad\text{and} \quad E_2(\eta)=\frac12\norm{\nabla^{k_0} \eta}_{L^2(Q)}^2.
\end{equation}

The convergence analysis here can also be applied to more abstract settings. For that one needs to assume besides Assumption~\ref{assumptions-general} that the estimates in Lemma~\ref{lem:ineq-DE1} and \Cref{lem:ineq-DE2} have to be satisfied. In particular this follows when additional to Assumption~\ref{assumptions-general} the following two assumptions are made:
\begin{enumerate}
\item[(A.6)] The limit solution $\eta$ satisfies
\begin{equation*}
\norm{\fint_0^a DE(\eta(t+s)) -DE(\eta(t))\dx[s]}_{X^*}\leq Ca
\end{equation*}
for all $t\in [0,T-a]$ and $a\in [0,\tau_0]$.
\item[(A.7)] For all $\alpha,\beta \in \mathcal E$, $\gamma \in X$ 
\begin{equation*}
\langle DE(\alpha)-DE(\beta)), \gamma\rangle\leq C\norm{\alpha-\beta}_X\norm{\gamma}_H
\end{equation*}
are satisfied.
\end{enumerate}
Please note that these assumptions are true for a large class of problems including the ODE examples in the numeric section.

\subsection{Time-regularity}
\label{ssec:timereg}
In order to indicate the validity of the convergence analysis we first include here some higher order a-priori estimates for smooth solutions. Our technique is to introduce the dissipation term of the form $\epsilon(-\Delta)^{k_0}\partial_t\eta_\epsilon$ and then remove the term after having obtained uniform estimates.

We have, for each $\epsilon>0$, given 
\begin{equation}\label{eqn:convrate-ic-rhs}
 \eta_0\in \mathcal{E},\quad \eta_* \in L^2(Q),\quad f\in L^2((0,T);L^2(Q)),
 \end{equation} a solution $\eta_\epsilon$ of 
\begin{align}
\begin{aligned}\label{eqn:eta-epsilon-eqn}
\partial_{tt}\eta_\epsilon + (-\Delta)^{k_0}\eta_\epsilon +\epsilon (-\Delta)^{k_0}\partial_t\eta_\epsilon -\div (\nabla_\xi e(\nabla\eta_\epsilon))&= f \\
\eta_\epsilon(t) &\in \mathcal E, \quad t\in(0,T)\\
\eta_\epsilon(0)&=\eta_0\\
\partial_t\eta_\epsilon (0) &=\eta_*\\
\eta_\epsilon(t,x)&= x,\quad x\in \Gamma_D\\
\partial_\nu \eta_\epsilon(t,x) &= \nu(x), \quad x\in \Gamma_N.
\end{aligned}
\end{align}
provided no collision happens in the time interval $(0,T)$.

By the previous, we have the existence of $\eta_\epsilon$ with \begin{equation*}
\eta_\epsilon\in L^\infty((0,T);W^{k_0,2}(Q)),\quad \partial_t\eta_\epsilon \in L^\infty((0,T);L^2(Q)),\quad \partial_t \eta_\epsilon \in L^2((0,T);W^{k_0,2}(Q))
\end{equation*}
satisfying the estimates
\begin{equation}\label{eqn:eta-epsilon-estimates}
\norm{\eta_\epsilon}_{L^\infty((0,T);W^{k_0,2}(Q))},\norm{\partial_t\eta_\epsilon}_{L^\infty((0,T);L^2(Q))}, \sqrt \epsilon\norm{\partial_t \eta_\epsilon}_{  L^2((0,T);W^{k_0,2}(Q))}\leq C(\eta_0,\eta_*,f),
\end{equation}
here and further $C(\dots)$ is a constant depending on the parameters in parenthesis.
Indeed one can use the approximation 
\begin{equation*}
\eta_k^{\ell} = \argmin_{\eta\in\mathcal E}\frac{\tau h}{2}\norm{\frac{\frac{\eta-\eta_{k-1}^\ell}{\tau}-\frac{\eta_k^{\ell-1}-\eta_{k-1}^{\ell-1}}{\tau}}{h}
}_{L^2}^2 +\epsilon\frac{\tau}{2}\norm{\nabla^{k_0}\frac{\eta-\eta_{k-1}^\ell}{\tau}}_{L^2}^2+E(\eta) -\langle f_k^\ell , \eta\rangle,
\end{equation*}
with the same initial conditions. Following step by step the argument in \Cref{thm:stability-solid} uniform stability estimates are available with the dissipation additionally appearing on the left hand side. Hence by the usual weak compactness results and compact embeddings a solution to \eqref{eqn:eta-epsilon-estimates} can be established by letting $\tau,h\to 0$.

We now investigate higher time regularity, given regularity of initial conditions and the right hand side. In particular we note which estimate do or do not depend on $\epsilon$. For ease of notation let us denote \begin{equation*}
W^{k_0,2}_D(Q):= \{\eta\in W^{k_0,2}(Q): \eta(x)=x\text{ for }x\in \Gamma_D\}.
\end{equation*} We begin with the following lemma.

\begin{lemma}\label{lem:beta-regularity}
The following problem for $\beta$ 
\begin{align*}
\begin{aligned}
\partial_{tt}\beta +(-\Delta)^{k_0} \beta +\epsilon (-\Delta)^{k_0}\partial_t\beta &= g ,\quad \text{in }(0,T) \\
\beta(0)&=\beta_0\\
\partial_t \beta(0)&=\beta_* 
\end{aligned}
\end{align*}
has, for given data
\begin{equation*}
g \in L^2((0,T);W^{-1,2}(Q)),\quad \beta_0\in W^{k_0,2}_D(Q),\quad \beta_* \in L^2(Q)
\end{equation*}
a unique solution $\beta$ with
\begin{equation*}
\beta \in L^\infty((0,T);W^{k_0,2}_D(Q)), \quad \partial_t\beta \in L^\infty((0,T);L^2(Q)),\quad \partial_t\beta \in L^2((0,T);W^{k_0,2}(Q))
\end{equation*}
satisfying the estimate
\begin{equation}\label{eqn:beta-estimate}
\norm{\beta}_{L^\infty((0,T);W^{k_0,2}(Q))}, \norm{\partial_t \beta}_{L^\infty((0,T);L^2(Q))}, \sqrt \epsilon \norm{\partial_t\beta}_{L^2((0,T);W^{k_0,2}(Q))}\leq C = C(\epsilon,\beta_0,\beta_*,g).
\end{equation}
\end{lemma}
\begin{proof}
We show the statement by Galerkin approximation.
Let $\{w_k\}_{k\in\N}\subset W^{k_0,2}(Q)$ be an orthogonal basis, moreover orthonormal in $L^2(Q)$. Then for $n\in \N$ we solve the following system of ODE
\begin{align*}
\alpha _k''(t)+\alpha_k(t)\norm{\nabla^{k_0} w_k}_{L^2}^2+\alpha_k'(t) \epsilon \norm{\nabla^{k_0} w_k}_{L^2}^2 &= \langle g,w_k\rangle, \quad k=1,\dots,n.\\
\alpha_k(0)&= \langle \beta_0,w_k\rangle_{L^2},\\
\alpha_k'(0)&= \langle\beta_*,w_k\rangle_{L^2}.
\end{align*}
The existence of absolutely continuous solutions $\alpha_k\colon (0,T)\to \R$ is standard theory of ODE. Then $\beta_n(t)=\sum_{k=1}^n \alpha_k(t)w_k$ solves the equation
\begin{align*}
\langle \partial_{tt}\beta_n, w_k\rangle_{L^2}+\langle\nabla^{k_0}\beta_n, \nabla^{k_0} w_k\rangle_{L^2}+\epsilon\langle\nabla^{k_0} \partial_t \beta_n, \nabla^{k_0} w_k \rangle_{L^2} &= \langle g, w_k\rangle ,\quad k=1,\dots,n, \\
\langle\beta_n(0),w_k\rangle_{L^2}&= \langle\beta_0,w_k\rangle_{L^2},\\
\langle\partial_t\beta_n(0),w_k\rangle_{L^2}&=\langle\beta_*,w_k\rangle_{L^2}.
\end{align*}
Now we multiply by $\alpha_k'$ and sum for $k=1,\dots,n$ (i.e. use $\partial_t \beta_n$ as a test function), so we obtain
\begin{equation*}
\frac12 \partial_t \norm{\partial_t \beta_n}_{L^2}^2+\frac12 \partial_t\norm{\nabla^{k_0}\beta_n}_{L^2}^2+\epsilon \norm{\nabla^{k_0}\partial_t\beta_n}_{L^2}^2 = \langle g,\partial_t\beta_n\rangle _{L^2}\leq \frac{1}{2\epsilon} \norm{g}_{W^{-1,2}}^2 + \frac{\epsilon}{2}\norm{\nabla \partial_t\beta_n}_{L^2}^2
\end{equation*}
Absorbing last term, using Gronwall inequality gives
\begin{equation*}
\frac12 \norm{\partial_t \beta_n}_{L^\infty((0,T);L^2)}^2+\frac12 \norm{\nabla^{k_0}\beta_n}_{L^\infty((0,T);L^2)}^2+\frac{\epsilon}{2} \norm{\nabla^{k_0}\partial_t\beta_n}_{L^2((0,T);L^2)}^2 \leq C(\epsilon,\beta_0,\beta_*,g). 
\end{equation*}
Passing $n\to \infty$ gives the result.
Finally, uniqueness of the solution can be readily seen from linearity.
\end{proof}

The lemma will now be used to show a better regularity of $\eta_\epsilon$.
\begin{lemma}\label{lem:epsilon-dependent}
Let it further be satisfied that \begin{equation}\label{eta0w2k0}
\eta_0\in W^{2k_0,2}(Q),\quad \eta_*\in W^{ 2k_0,2}(Q), \quad \partial_t f\in W^{1,2}((0,T);L^2(Q)).
\end{equation}
Let $\eta_\epsilon$ be a solution of 
\begin{align*}
\begin{aligned}
\partial_{tt}\eta_\epsilon + (-\Delta)^{k_0}\eta_\epsilon +\epsilon (-\Delta)^{k_0}\partial_t\eta_\epsilon -\div (\nabla_\xi e(\nabla\eta_\epsilon))&= f,\\
\eta_\epsilon(0)&=\eta_0,\\
\partial_t\eta_\epsilon (0) &=\eta_*.
\end{aligned}
\end{align*}
Then it holds
\begin{equation*}
\partial_t \eta_\epsilon \in L^\infty((0,T);W^{k_0,2}(Q)), \quad \partial_{tt}\eta_\epsilon \in L^\infty((0,T);L^2(Q)),\quad \partial_{tt}\eta_\epsilon \in L^2((0,T);W^{k_0,2}(Q))
\end{equation*}
with the estimate
\begin{equation*}
\norm{\partial_t\eta_\epsilon}_{L^\infty((0,T);W^{k_0,2}(Q))}, \norm{\partial_{tt} \eta_\epsilon}_{L^\infty((0,T);L^2(Q))}, \sqrt \epsilon \norm{\partial_{tt}\eta_\epsilon}_{L^2((0,T);W^{k_0,2}(Q))}\leq  C(\epsilon,\eta_0,\eta_*,f).
\end{equation*}

\end{lemma}
\begin{proof}
Consider the problem for $\beta$ 
\begin{align}
\begin{aligned}\label{eqn:beta-inproof}
\partial_{tt}\beta +(-\Delta)^{k_0} \beta +\epsilon (-\Delta)^{k_0}\partial_t\beta &= \partial_t f +\div(\nabla_\xi e(\nabla\eta_\epsilon)\nabla\partial_t\eta_\epsilon) \\
\beta(0)&=\eta_*\\
\partial_t \beta(0)&=f(0)-(-\Delta)^{k_0}\eta_0- \epsilon (-\Delta)^{k_0}\eta_*+\div(\nabla_\xi e(\nabla \eta_0)).
\end{aligned}
\end{align}
We are in a position to use \Cref{lem:beta-regularity} with
\begin{equation*}
g=\partial_t f +\div(\nabla_\xi e(\nabla\eta_\epsilon)\nabla\partial_t\eta_\epsilon),\quad \beta_0=\eta_*,\quad \beta_*= f(0)-(-\Delta)^{k_0}\eta_0- \epsilon (-\Delta)^{k_0}\eta_*+\div(\nabla_\xi e(\nabla \eta_0)).
\end{equation*}
Note that $g$, $\beta_0$, $\beta_*$ lie in the correct spaces to apply \Cref{lem:beta-regularity}, as $\nabla_\xi e(\nabla\eta_\epsilon)\in L^\infty((0,T)\times Q)$ and $\nabla\partial_t\eta_\epsilon \in L^2((0,T)\times Q)$ implies $g\in L^2((0,T);W^{-1,2}(Q))$, and also assumptions \eqref{eta0w2k0} imply $\beta_*\in L^2(Q)$. Thus \Cref{lem:beta-regularity} gives existence of $\beta$ with the respective estimates \eqref{eqn:beta-estimate}.

Now we need to check that $\beta=\partial_t\eta_\epsilon$. For this we define 
\begin{equation*}
\tilde \eta_\epsilon(t):= \eta_0 + \int_0^t \beta\dt.
\end{equation*}
Clearly $\partial_t\tilde{\eta}_\epsilon=\beta$ and now we need to check that $\tilde{\eta_\epsilon}=\eta_\epsilon$. We will show this by arguing that $\tilde{\eta_\epsilon}$ solves the same linear equation as $\eta_\epsilon$.

For this we integrate the equation \eqref{eqn:beta-inproof} over $(0,t)$ to obtain
\begin{multline*}
\partial_t\beta(t)-\partial_t\beta(0)+ (-\Delta)^{k_0}\tilde \eta_\epsilon(t)-(-\Delta)^{k_0} \eta_0 + \epsilon(-\Delta^{k_0})\beta(t)-\epsilon(-\Delta^{k_0})\beta_0 \\
= f(t)-f(0) + \div(\nabla_\xi e(\nabla\eta_\epsilon(t))- \div(\nabla_\xi e(\nabla\eta_0)).
\end{multline*}
Using the initial conditions on $\beta$ we thus conclude that $\nabla\eta_\epsilon$ solves the linear equation with initial conditions
\begin{align*}
\partial_{tt}\tilde{\eta_\epsilon} + (-\Delta)^{k_0} \tilde{\eta_\epsilon} +\epsilon(-\Delta^{k_0})\partial_t\tilde{\eta_\epsilon} &= f +\div(\nabla_\xi e(\nabla\eta_\epsilon)) \\
\tilde{\eta_\epsilon}(0)&=\eta_0\\
\partial_t\tilde{\eta_\epsilon}(0)&=\eta_*.
\end{align*}
Since $\eta_\epsilon$ solves the same linear equation and its solutions are unique, it follows that $\eta_\epsilon=\tilde{\eta_\epsilon}$. Thus $\beta =\partial_t \eta_\epsilon$ and the proof is finished.
\end{proof}

We now can prove the time regularity estimate for $\eta_\epsilon$, which is \emph{independent} of $\epsilon$.

\begin{theorem}[Time regularity]
\label{thm:timereg}
Let it further be satisfied \begin{equation*}
\eta_0\in W^{2k_0,2}(Q),\quad \eta_*\in W^{ 2 k_0,2}(Q), \quad \partial_t f\in L^2((0,T);L^2(Q)).
\end{equation*}
Then the solution $\eta_\epsilon$ of \eqref{eqn:eta-epsilon-eqn} satisfies
\begin{equation*}
\partial_{tt}\eta_\epsilon \in L^\infty((0,T);L^2(Q)), \quad \partial_t \eta_\epsilon\in L^\infty((0,T);W^{k_0,2}(Q))
\end{equation*}
with the $\epsilon$-independent estimate
\begin{equation*}
\norm{\partial_{tt}\eta_\epsilon}_{L^\infty((0,T);L^2(Q))}+ \norm{\partial_t \eta_\epsilon}_{L^\infty((0,T);W^{k_0,2}(Q))}\leq C = C(\eta_0,\eta_*,f).
\end{equation*}
\end{theorem}

\begin{proof}
As shown during the proof of \Cref{lem:epsilon-dependent}, $\eta_\epsilon$ satisfies the equation
\begin{align*}
\partial_{ttt}\eta_\epsilon +(-\Delta)^{k_0} \partial_t \eta_\epsilon +\epsilon (-\Delta)^{k_0}\partial_{tt}\eta_\epsilon &= \partial_t f + \div(\nabla_\xi e(\nabla\eta_\epsilon)\nabla \partial_t\eta_\epsilon) \\
\eta_\epsilon(0)&=\eta_0\\
\partial_t\eta_\epsilon(0)&=\eta_*\\
\partial_{tt} \eta_\epsilon(0)&= f(0)-(-\Delta)^{k_0}\eta_0- \epsilon(-\Delta)^{k_0}\eta_*+\div(\nabla_\xi e(\nabla \eta_0))
\end{align*}
By the same lemma we have  $\partial_{tt}\eta_\epsilon \in L^\infty((0,T);L^2(Q))\cap L^2((0,T);W^{k_0,2}(Q))$ so we can use $\partial_{tt}\eta_\epsilon$ as a test function. This yields 
\begin{align*}
\frac{1}{2}\partial_t\norm{\partial_{tt}\eta_\epsilon}_{L^2}^2 + \frac{1}{2}\partial_t \norm{\nabla^{k_0}\partial_t\eta_\epsilon}_{L^2}^2+\epsilon\norm{\nabla^{k_0}\partial_{tt}\eta_\epsilon}_{L^2}^2 = \langle \partial_t f, \partial_{tt} \eta_\epsilon\rangle +\langle\div(\nabla_\xi e(\nabla\eta_\epsilon)\nabla \partial_t\eta_\epsilon), \partial_{tt}\eta_\epsilon\rangle \\ \leq \norm{\partial_t f + \div(\nabla_\xi e(\nabla\eta_\epsilon)\nabla \partial_t\eta_\epsilon)}_{L^2} \norm{\partial_{tt}\eta_\epsilon}_{L^2}.
\end{align*}
We now apply the chain rule (for which we have enough regularity by above) 
\begin{equation*}
\div(\nabla_\xi e(\nabla\eta_\epsilon)\nabla \partial_t\eta_\epsilon)= \nabla_{\xi}^2 e(\nabla\eta_\epsilon) \nabla^2\eta_\epsilon \nabla\partial_t\eta_\epsilon + \nabla_\xi e(\nabla\eta_\epsilon) \nabla^2 \partial_t\eta_\epsilon.
\end{equation*}
Now estimate by H\"older inequality and the Sobolev inequality for the continuous embedding $W^{k_0,2}(Q)\subset W^{1,\infty}(Q)$ and the Poincar\'e inequality (here $c_s$ resp. $c_p$ is the Sobolev resp. Poincar\'e constant)
\begin{multline*}
\norm{\div(\nabla_\xi e(\nabla\eta_\epsilon)\nabla \partial_t\eta_\epsilon)}_{L^2} \leq \norm{\nabla^2_\xi e(\nabla\eta_\epsilon)}_{L^\infty}\! \norm{\nabla^2\eta_\epsilon}_{L^2}\norm{\nabla\partial_t\eta_\epsilon}_{L^\infty}\! + \norm{\nabla_\xi e(\nabla\eta_\epsilon)}_{L^\infty }\!\norm{\nabla^2\partial_t\eta_\epsilon}_{L^2} \\ \leq (c_s \norm{\nabla_\xi^2 e(\nabla\eta_\epsilon)}_{L^\infty} \norm{\nabla^2\eta_\epsilon}_{L^2}+ c_p\norm{\nabla_\xi e(\eta_\epsilon)}_{L^\infty} ) \norm{\nabla^{k_0}\partial_t\eta_\epsilon}_{L^2}
\end{multline*}
Note that by energy estimates \eqref{eqn:eta-epsilon-estimates} and \Cref{lem:e-density-bounded} we have $$\norm{\nabla^2\eta_\epsilon}_{L^2}, \norm{\nabla^2 e(\nabla\eta_\epsilon)}_{L^\infty} \leq C=C(E(\eta_0),\norm{\eta_*}_{L^2},f).$$ Therefore we have 
\begin{equation*}
\frac12 \left(\partial_t\norm{\partial_{tt}\eta_\epsilon}_{L^2}^2+\partial_t\norm{\nabla^{k_0}\partial_t\eta_\epsilon}_{L^2}^2\right) +\epsilon\underbrace{\norm{\nabla^{k_0} \partial_{tt}\eta_\epsilon}_{L^2}^2}_{\geq 0} \leq C\left(1+\norm{\partial_{tt}\eta_\epsilon}_{L^2}^2+\norm{\nabla^{k_0}\partial_t\eta_\epsilon}_{L^2}^2\right)
\end{equation*}
where $C$ is independent of $\epsilon$. Applying Gronwall inequality to  $\norm{\partial_{tt}\eta_\epsilon}_{L^2}^2+\norm{\nabla^{k_0}\partial_t\eta_\epsilon}_{L^2}^2$ we obtain the desired estimate
\begin{equation*}
\norm{\partial_{tt}\eta_\epsilon}_{L^\infty((0,T);L^2(Q))}^2+\norm{\nabla^{k_0}\partial_t\eta_\epsilon}_{L^\infty((0,T);L^2(Q))}^2 \leq e^T\left (\norm{\eta_*}_{L^2}^2 +\norm{\nabla^{k_0} \eta_0}_{L^2}^2 +CT\right ),
\end{equation*}
where the right hand side depends on $\eta_0,\eta_*,f$.
\end{proof}

\begin{theorem}[Higher time regularity]
\label{thm:highertime}
Assume it holds
\begin{equation*}
\eta_0, \eta_*\in W^{3k_0,2}(Q), \quad f \in W^{2,2}((0,T); L^2(Q)).
\end{equation*}
and further assume that $\norm{\Delta^{2k_0}\eta_*}_{L^2(Q)}\leq \frac{1}{\epsilon}$,
then
\begin{equation*}
\partial_{tt} \eta_\epsilon \in L^\infty((0,T);W^{k_0,2}(Q)), \quad \partial_{ttt}\eta_\epsilon \in L^\infty((0,T);L^2(Q)),\quad \partial_{ttt}\eta_\epsilon \in L^2((0,T);W^{k_0,2}(Q)).
\end{equation*}
with the $\epsilon$-independent estimate 
\begin{equation*}
\norm{\partial_{tt} \eta_\epsilon}_{L^\infty((0,T);W^{k_0,2}(Q))}+\norm{\partial_{ttt}\eta_\epsilon }_{ L^\infty((0,T);L^2(Q))}+\sqrt{\epsilon}\norm{\partial_{ttt}\eta_\epsilon}_{L^2((0,T);W^{k_0,2}(Q))}\leq C(\eta_0,\eta_*,f).
\end{equation*}
\end{theorem}
\begin{proof}
This follows by an iteration of the previous proof. Use
\begin{align*}
g&=\partial_{tt} f + \partial_{tt} \div(\nabla_\xi e(\nabla\eta_\epsilon)),\\
\beta_0&=f(0)-(-\Delta)^{k_0}\eta_0- \epsilon(-\Delta)^{k_0}\eta_*+\div(\nabla_\xi e(\nabla \eta_0)),\\
\beta_* &= \partial_tf(0)-(-\Delta)^{k_0}\eta_* - \epsilon(-\Delta)^{k_0}\beta_0 +\div(\nabla_\xi^2 e(\nabla\eta_0)\nabla \eta_*),
\end{align*}
  we can readily verify the assumptions of \Cref{lem:beta-regularity}, that is $g\in L^2((0,T);W^{-1,2}(Q))$, $\beta_0\in W^{k_0,2}(Q)$, $\beta_*\in L^2(Q)$. Proceed as before to see that $\beta=\partial_{tt}\eta_\epsilon$ has the given regularity with given estimates.
\end{proof}

\subsubsection*{No dissipation limit}
We will now see that when the dissipation vanishes, we obtain as the limit a~solution to the problem without dissipation.
\begin{theorem}\label{thm:higher-space-regularity} 
For $\eta_0\in W^{k_0,2}_D(Q)$, $\eta_*\in L^2(Q)$, $f\in L^2((0,T);L^2(Q))$. Then
there exists a weak solution to
\begin{align}
\begin{aligned}\label{eqn:eta-regularity-eqn}
\partial_{tt}\eta + (-\Delta)^{k_0}\eta -\div (\nabla_\xi e(\nabla\eta))&= f \\
\eta(0)&=\eta_0\\
\partial_t\eta (0) &=\eta_*.
\end{aligned}
\end{align}
If further 
 \begin{equation*}
\eta_0\in W^{2k_0,2}(Q),\quad \eta_*\in W^{ 2 k_0,2}(Q), \quad \partial_t f\in L^2((0,T);L^2(Q)),
\end{equation*}
then 
\begin{equation*}
\partial_{tt}\eta \in L^\infty((0,T);L^2(Q)), \quad \partial_t \eta\in L^\infty((0,T);W^{k_0,2}(Q)), \quad \eta\in L^\infty((0,T);W^{2k_0,2}_{\text{loc}}(Q))
\end{equation*}
with
\begin{equation*}
\norm{\partial_{tt}\eta}_{L^\infty((0,T);L^2(Q))}+ \norm{\partial_t \eta}_{L^\infty((0,T);W^{k_0,2}(Q))}+\norm{\Delta^{k_0} \eta}_{L^\infty((0,T);L^2(Q))}\leq  C(\eta_0,\eta_*,f).
\end{equation*}
If further,
\begin{equation*}
\eta_0, \eta_*\in W^{3k_0,2}(Q), \quad f \in W^{2,2}((0,T); L^2(Q)),
\end{equation*}
then
\begin{equation*}
\partial_{tt} \eta \in L^\infty((0,T);W^{k_0,2}(Q)), \quad \partial_{ttt}\eta \in L^\infty((0,T);L^2(Q)),\quad  \partial_{t}\eta\in L^\infty((0,T);W^{2k_0,2}_{\text{loc}}(Q))
\end{equation*}
with
\begin{equation}
\label{eqn:time-reg-estimates}
\norm{\partial_{ttt}\eta}_{L^\infty((0,T);L^2(Q))}+ \norm{\partial_{tt} \eta}_{L^\infty((0,T);W^{k_0,2}(Q))}+\norm{\Delta^{k_0} \partial_{t}\eta}_{L^\infty((0,T);L^2(Q))}\leq  C(\eta_0,\eta_*,f).
\end{equation}
\end{theorem}
\begin{proof}
Let for $\epsilon>0$, be $\eta_\epsilon$ a~solution of
\begin{align*}
\partial_{tt}\eta_\epsilon + (-\Delta)^{k_0}\eta_\epsilon +\epsilon (-\Delta)^{k_0}\partial_t\eta_\epsilon -\div (\nabla_\xi e(\nabla\eta_\epsilon))&= f\\
\eta_\epsilon(0)&=\eta_0\\
\partial_t\eta_\epsilon (0) &=\eta_*.
\end{align*}
Then for (some subsequence of) $\epsilon \to 0$ we have 
\begin{equation*}
\eta_\epsilon \weakto^*  \eta \text{ in }
L^\infty((0,T);W^{k_0,2}(Q)),\quad \partial_t\eta_\epsilon \weakto^* \partial_t\eta\text{ in }L^\infty((0,T);L^2(Q)),
\end{equation*}
From the estimates \eqref{eqn:eta-epsilon-estimates} we see that we can take a converging subsequence 
\begin{equation*}
\eta_\epsilon \weakto^*  \eta \text{ in }
L^\infty((0,T);W^{k_0,2}(Q)),\quad \partial_t\eta_\epsilon \weakto^* \partial_t\eta\text{ in }L^\infty((0,T);L^2(Q)),
\end{equation*}
and moreover, by the Aubin-Lions lemma we see that (for a further subsequence)
\begin{equation}\label{eqn:eta-eps-strongly-C1a}
\eta_\epsilon\to \eta \text{ in }C([0,T];C^{1,\alpha}(Q)).
\end{equation}
It remains to check that $\eta$ solves the equation. Let us test the equation for $\eta_\epsilon$ with a test function $\varphi\in C^\infty_c([0,T)\times Q)$. Then we see
\begin{equation*}
 \langle \eta_*, \varphi(0)\rangle + \int_0^T\hspace{-1em} -\langle \partial_t \eta_\epsilon, \partial_t\varphi \rangle + \langle\nabla^{k_0} \eta_\epsilon ,\nabla^{k_0} \varphi \rangle + \varepsilon\langle \nabla^{k_0} \partial_t \eta_\epsilon ,  \nabla^{k_0} \varphi \rangle + \langle\nabla_\xi e(\nabla\eta_\epsilon),\nabla \varphi\rangle \dx[t] = \int_0^T \langle f,\varphi\rangle \dx[t].
\end{equation*}
Limit passage in the first two terms on the left is from the weak convergence, and in the last term on the left from the strong convergence \eqref{eqn:eta-eps-strongly-C1a}. Finally, regarding the dissipation term we estimate 
\begin{equation*}
\left| \int_0^T \epsilon\langle \nabla^{k_0} \partial_t \eta_\epsilon ,  \nabla^{k_0} \varphi \rangle \dx[t]\right| \leq  \sqrt \epsilon\sqrt\epsilon\norm{\nabla^{k_0} \partial_t\eta_\epsilon}_{L^2((0,T);L^2(Q))} \norm{\nabla^{k_0}\varphi}_{L^2((0,T);L^2(Q))}\to 0,
\end{equation*}
since by the \eqref{eqn:eta-epsilon-estimates}, $\sqrt\epsilon\norm{\nabla^{k_0} \partial_t\eta_\epsilon}_{L^2((0,T);L^2(Q))}$ is bounded independent of $\epsilon$.
So passing to the limit $\epsilon\to 0$ we have 
\begin{equation*}
 \langle \eta_*, \varphi(0)\rangle + \int_0^T -\langle \partial_t \eta, \partial_t\varphi \rangle + \langle\nabla^{k_0} \eta,\nabla^{k_0} \varphi \rangle + \langle\nabla_\xi e(\nabla\eta),\nabla \varphi\rangle \dx[t] = \int_0^T \langle f,\varphi\rangle \dx[t].
\end{equation*}
Further, the initial condition $\eta(0)=\eta_0$ is satisfied due to $\eta_\epsilon(0)=\eta_0$ and the strong convergence \eqref{eqn:eta-eps-strongly-C1a}.

The regularity estimates except for the $\Delta^{k_0}\eta$ and $\Delta^{k_0}\partial_t\eta$ follow by the uniform in $epsilon$ estimates in \Cref{thm:timereg} and \Cref{thm:highertime}. Further they imply in the first case that $f(t) -\partial_{tt}\eta(t) +\div(\nabla_\xi e(\nabla\eta(t)))\in L^\infty((0,T);L^2(Q))$,
which implies that 
\[
(-\Delta)^{k_0} \eta(t) =f(t) -\partial_{tt}\eta(t) +\div(\nabla_\xi e(\nabla\eta(t)))
\]
almost everywhere and the estimate of $\Delta^{k_0}\eta$

Similar we find the second case that $\partial_t f(t) -\partial_{ttt}\eta(t) +\div(\nabla_\xi^2 e(\nabla\eta(t))\nabla\partial_t\eta(t))$ is bounded in $L^\infty((0,T);L^2(Q))$. Hence
\begin{equation*}
(-\Delta)^{k_0} \partial_t\eta(t) = \partial_t f(t) -\partial_{ttt}\eta(t) +\div(\nabla_\xi^2 e(\nabla\eta(t))\nabla\partial_t\eta(t)),
\end{equation*}
is satisfied almost everywhere and the estimate of $\Delta^{k_0}\partial_t\eta$ follows.

The regularity in $W^{2k_0,2}_{\text{local}}(Q)$ follows by the local regularity theory for the $(k_0)$-Laplace equation (which follows by applying iteratively the local theory for the Poisson equation). 
\end{proof}

\begin{remark}[Space regularity]
Please observe that in many situations the fact that $\Delta^{k_0}\eta\in L^2(Q)$ implies global higher spacial regularity up to $\eta\in W^{2k_0,2}(Q)$. This regularity however sensitively depends on the regularity and shape of the domain. For the sake of the generality of domains and boundary values we decided to not make this specific here. Certainly, local estimates are always available by the classical theory for the Poisson equation. 
\end{remark}

\begin{remark}[Improvement of regularity with dissipation]
In case that we equation includes a dissipation term $(-\Delta)^{k_0}\partial_t\eta$, we can observe a time regularizing effect known from parabolic equations. Therefore in that case, it is possible to start with initial data of no higher regularity, namely only \eqref{eqn:convrate-ic-rhs}. By a suitable testing of the equation we get that the regularity improves in an arbitrarily short time interval, and therefore it is possible to take this new time as initial and perform the procedure above.
\end{remark}

\subsection{Convergence rate -- elastic solid} Now we turn to the main question of this section, the rate of convergence of our scheme. 
In this section we again, to ease the notation, adopt the convention that $\|\cdot\|$ without any index is the $L^2(Q)$-norm, resp. $\langle\cdot,\cdot\rangle$ is the $L^2(Q)$-scalar product (or dual pairing in the $DE_1$ terms). Moreover we strengthen the assumption \eqref{as:e-C2} so that $e$ has one more derivative:
\begin{enumerate}
\item[(E.1')] $e\in {C}^3(\R^{n\times n}_{\det > 0})$, where $\R^{n\times n}_{\det > 0}= \{M\in\R^{n\times n }: \det M > 0\}$\label{as:e-C3}.
\end{enumerate}
The solution $\eta$ solves the equation
\begin{equation}\label{eqn:eta-convrate-eqn}
\partial_{tt}\eta + DE(\eta)=f
\end{equation}
To compare $\eta$ with our minimizing movements approximation $\{\eta_k^\ell\}_{k,\ell}$, we will make a discrete version of $\eta$.
For this let us integrate \eqref{eqn:eta-convrate-eqn} over time $s\in (t-h,t)$ and then $t\in(t_{k-1}^\ell,t_k^\ell)$ and divide by $\tau$ and $h$, to get (recall the definition \eqref{eqn:fkl-def} of $f_k^\ell$)
\begin{equation}\label{eqn:eta-discretized-eqn}
\frac{\frac{\eta(t_k^\ell)-\eta(t_{k-1}^\ell)}{\tau}-\frac{\eta(t_k^{\ell-1})-\eta(t_{k-1}^{\ell-1})}{\tau}}{h} + \fint_0^\tau \fint_0^h DE(\eta(t_{k-1}^\ell+s+\sigma)) \dx[s]\dx[\sigma]=f_k^\ell
\end{equation}
Please recall that we have introduced two schemes for construction of $\eta_k^\ell$, one with stabilization term \eqref{def:eta-minimizing-movements} and one without it \eqref{def:eta-minimizing-movements-nodiss}. All calculations are made here for $\eta_k^\ell$ with the dissipation term, that is from the approximation \eqref{def:eta-minimizing-movements}. We remark here that the version without stabilization \eqref{def:eta-minimizing-movements-nodiss}, the only difference is that the term $-c\tau \Delta \frac{\eta_k^\ell-\eta_{k-1}^\ell}{\tau}$ will not appear. Thus all the calculations below go through (slightly more simply) also in this case, and we get the same result.

Let us thus recall the Euler-Lagrange equation for $\eta_k^\ell\in \operatorname{int}\mathcal E$ (recall that we exclude collisions): 
\begin{equation}\label{eqn:eta-discrete-EL-eqn}
\frac{\frac{\eta_k^\ell-\eta_{k-1}^\ell}{\tau}-\frac{\eta_k^{\ell-1}-\eta_{k-1}^{\ell-1}}{
\tau}}{h} + DE(\eta_k^\ell) -c\tau \Delta \frac{\eta_k^\ell-\eta_{k-1}^\ell}{\tau} = f_k^\ell.
\end{equation}

Subtract \eqref{eqn:eta-discrete-EL-eqn} and \eqref{eqn:eta-discretized-eqn}, denote the error term by $e_k^\ell = \eta_k^\ell - \eta(t_k^\ell)$ and get
\begin{equation*}
\frac{\frac{e_k^\ell-e_{k-1}^\ell}{\tau}-\frac{e_k^{\ell-1}-e_{k-1}^{\ell-1}}{\tau}}{h}-c\tau \Delta \frac{\eta_k^\ell-\eta_{k-1}^\ell}{\tau}+\fint_0^\tau\fint_0^h DE(\eta_k^\ell)-DE(\eta(t_{k-1}^\ell+s+\sigma)) \dx[s]\dx[\sigma]=0
\end{equation*}
We add and subtract $DE(\eta(t_k^\ell))-\Delta \frac{\eta(t_k^\ell)-\eta(t_{k-1}^\ell)}{\tau}$, so that
\begin{multline*}
\frac{\frac{e_k^\ell-e_{k-1}^\ell}{\tau}-\frac{e_k^{\ell-1}-e_{k-1}^{\ell-1}}{\tau}}{h}+ DE(\eta_k^\ell)-DE(\eta(t_{k}^\ell))-c\tau \Delta \frac{e_k^\ell-e_{k-1}^\ell}{\tau}\\ =\fint_0^\tau\fint_0^h DE(\eta(t_{k-1}^\ell+s+\sigma)) -DE(\eta(t_k^\ell))\dx[s]\dx[\sigma] + \Delta \frac{\eta(t_k^\ell)-\eta(t_{k-1}^\ell)}{\tau}.
\end{multline*}
Use as a test function $\frac{e_k^\ell-e_{k-1}^\ell}{\tau}$ to obtain
\begin{multline*}
\frac{1}{2h}\norm{\frac{e_k^\ell-e_{k-1}^\ell}{\tau}}^2 + \frac{1}{2h}\norm{\frac{e_k^\ell-e_{k-1}^\ell}{\tau}-\frac{e_k^{\ell-1}-e_{k-1}^{\ell-1}}{\tau}}^2 +\left\langle \nabla^{k_0} e_k^\ell,\nabla^{k_0}\frac{e_k^\ell-e_{k-1}^\ell}{\tau}\right\rangle  +c\tau \left\| \nabla\frac{e_k^\ell-e_{k-1}^\ell}{\tau} \right\|^2 \\ = \frac{1}{2h}\norm{\frac{e_k^{\ell-1}-e_{k-1}^{\ell-1}}{\tau}}^2 -\left\langle DE_1(\eta_k^\ell)-DE_1(\eta(t_k^\ell)), \frac{e_k^\ell-e_{k-1}^\ell}{\tau}\right\rangle
\\  +\left\langle \fint_0^\tau\fint_0^h DE(\eta(t_{k-1}^\ell+s+\sigma)) -DE(\eta(t_k^\ell))\dx[s]\dx[\sigma],\frac{e_k^\ell-e_{k-1}^\ell}{\tau} \right\rangle -c\tau \left\langle\nabla \frac{\eta(t_k^\ell)-\eta(t_{k-1}^\ell)}{\tau}, \nabla\frac{e_k^\ell-e_{k-1}^\ell}{\tau} \right\rangle
\end{multline*}
so that
\begin{multline}\label{eqn:convrate-aftertesting}
\frac{1}{2h}\norm{\frac{e_k^\ell-e_{k-1}^\ell}{\tau}}^2 + \frac{1}{2h}\norm{\frac{e_k^\ell-e_{k-1}^\ell}{\tau}-\frac{e_k^{\ell-1}-e_{k-1}^{\ell-1}}{\tau}}^2 + \frac{1}{2\tau}\norm{\nabla^{k_0} e_k^\ell}^2 + \frac{1}{2\tau} \norm{\nabla^{k_0} e_k^\ell-\nabla^{k_0} e_{k-1}^\ell}^2
\\ +c\tau \left\| \nabla\frac{e_k^\ell-e_{k-1}^\ell}{\tau} \right\|^2 = \frac{1}{2h}\norm{\frac{e_k^{\ell-1}-e_{k-1}^{\ell-1}}{\tau}}^2 + \frac{1}{2\tau}\norm{\nabla^{k_0} e_{k-1}^\ell}^2 -\left\langle DE_1(\eta_k^\ell)-DE_1(\eta(t_k^\ell)), \frac{e_k^\ell-e_{k-1}^\ell}{\tau}\right\rangle
 \\ +\left\langle \fint_0^\tau\fint_0^h DE(\eta(t_{k-1}^\ell+s+\sigma)) -DE(\eta(t_k^\ell))\dx[s]\dx[\sigma],\frac{e_k^\ell-e_{k-1}^\ell}{\tau} \right\rangle-c\tau \left\langle\nabla \frac{\eta(t_k^\ell)-\eta(t_{k-1}^\ell)}{\tau}, \nabla\frac{e_k^\ell-e_{k-1}^\ell}{\tau} \right\rangle.
\end{multline}
Therefore we need to estimate the two terms containing the difference of energies.

\begin{lemma}\label{lem:ineq-DE1}
There exist $C_1,C_2$ depending on the energy bound of \Cref{thm:stability-solid} such that
\begin{equation*}
\langle DE_1(\eta_k^\ell)-DE_1(\eta(t_k^\ell)), e_k^\ell-e_{k-1}^\ell \rangle\leq \tau C_1\norm{\nabla^{k_0} e_k^\ell}^2+ \tau C_2\norm{ \frac{e_k^\ell-e_{k-1}^\ell}{\tau}}^2
\end{equation*}
\end{lemma}
\begin{proof}
Write 
\begin{equation}
\langle DE_1(\eta_k^\ell)-DE_1(\eta(t_k^\ell)), e_k^\ell-e_{k-1}^\ell \rangle = \int_Q (\nabla_\xi e(\nabla\eta_k^\ell)-\nabla_\xi e(\nabla\eta(t_k^\ell)):\nabla(e_k^\ell-e_{k-1}^\ell)\dx
\end{equation}
and denote $\eta_\theta=\theta \eta_k^\ell +(1-\theta)\eta(t_k^\ell)$ for $\theta \in [0,1]$. We modify the energy density $e$ so that it remains bounded on $\nabla \eta_\theta$ over $Q$.

Let us take $\tilde e\colon \R^{n\times n}\to \R$, $\tilde e\in C^3(\R^{n\times n})$, $\tilde e(\xi)= e(\xi)$ for $\xi\in\R^{n\times n}$, $\det \xi \geq \epsilon_0$, $|\xi|\leq C$, where $\epsilon_0$ is the lower bound on the determinant \eqref{as:det-lower-bound} and $C$ is from the energy estimate of \Cref{thm:stability-solid}, such that $\norm{\tilde e}_{C^3(\R^{n\times n})}$ depends only on these constants. Such $\tilde e$ can be constructed e.g. as 
\begin{equation*}
\tilde e = e\cdot (\chi_{\R^{n\times n}_{\det \geq \epsilon_0}+B_\delta(0)}*\psi_\delta), \text{ where }2\delta=\operatorname{dist}\bigl(\R^{n\times n}_{\det \leq 0},\{\xi\in\R^{n\times n}: \det\xi\geq \epsilon_0, |\xi|\leq C\} \bigr),
\end{equation*}
where $\chi_M$ is the characteristic function of $M$, $\psi_\delta$ is the standard mollifier and $*$ denotes convolution.

Then we have $e(\nabla \eta_k^\ell)=\tilde e(\nabla\eta_k^\ell)$  and  $e(\nabla \eta(t_k^\ell ))=\tilde e(\nabla\eta(t_k^\ell))$. Thus, since $\tilde e$ is everywhere finite and $C^3$, we can write
\begin{equation*}
\langle DE_1(\eta_k^\ell)-DE_1(\eta(t_k^\ell)), e_k^\ell-e_{k-1}^\ell \rangle = \int_Q \int_0^1 \hspace{-.9em}\nabla_\xi^2 \tilde e(\nabla(\theta \eta_k^\ell +(1-\theta)\eta(t_k^\ell))) \dx[\theta] : \nabla e_k^\ell :\nabla (e_k^\ell-e_{k-1}^\ell) \dx[\theta] \dx
\end{equation*}
Integrating by parts, this gives 
\begin{equation*}
= -\int_Q \int_0^1 \nabla_\xi^3 \tilde e (\nabla \eta_\theta) \nabla^2\eta_\theta  \dx[\theta] : \nabla e_k^\ell : (e_k^\ell-e_{k-1}^\ell) \dx - \int_Q \int_0^1 \nabla_\xi^2 \tilde e(\nabla \eta_\theta) \dx[\theta] : \nabla^2 e_k^\ell \cdot (e_k^\ell-e_{k-1}^\ell)\dx.
\end{equation*}

Hence we can use the H\"older inequality and obtain 
\begin{multline*}
\langle DE_1(\eta_k^\ell)-DE_1(\eta(t_k^\ell)), e_k^\ell-e_{k-1}^\ell \rangle\\
\leq \int_0^1 \norm{\nabla_\xi^3 \tilde e(\nabla\eta_\theta)}_{L^\infty}\!\norm{\nabla^2\eta_\theta} \norm{\nabla e_k^\ell}_{L^\infty}\!\norm{e_k^\ell-e_{k-1}^\ell} + \norm{\nabla_\xi^2\tilde e(\nabla\eta_\theta))}_{L^\infty}\norm{\nabla^2e_k^\ell}\norm{e_k^\ell-e_{k-1}^\ell}\! \dx[\theta].
\end{multline*}
We have, by our estimates, ($C$ independent of time)
\begin{equation*}
\norm{\nabla^3_\xi \tilde e(\nabla\eta_\theta)}_{L^\infty}, \norm{\nabla^2 \eta_\theta} \leq C
\end{equation*}
and by the embedding $W^{1,\infty}(Q)\subset W^{k_0,2}(Q)$
\begin{equation*}
\norm{\nabla e_k^\ell}_{L^\infty}\leq C \norm{\nabla^{k_0} e_k^\ell}
\end{equation*}
which in total, after Poincar\'e inequality $\norm{\nabla^2 e_k^\ell} \leq C\norm{\nabla^{k_0} e_k^\ell}$ gives
\begin{equation*}
\langle DE_1(\eta_k^\ell)-DE_1(\eta(t_k^\ell)), e_k^\ell-e_{k-1}^\ell \rangle\leq C\norm{\nabla^{k_0}e_k^\ell}\norm{e_k^\ell-e_{k-1}^\ell} \leq C_1\tau \norm{\nabla^{k_0} e_k^\ell}^2 + C_2 \tau\norm{\frac{e_k^\ell-e_{k-1}^\ell}{\tau}}^2
\end{equation*}
with constants $C_1,C_2$ independent of $\tau$.
\end{proof}

\begin{remark}
A more explicit construction of $\tilde e$ can be obtained in the case that $e$ explicitly depends on $\det \nabla \eta$. Suppose that $e\in C^3(\R^{n\times n}_{\det > 0})$ is of the form 
\begin{equation*}
e(\xi)= h(\xi, \det \xi)\quad \text{ with }\quad h\in C^3(\R^{n\times n}\times (0,\infty)).
\end{equation*}
Then we can truncate 
\begin{equation*}
\tilde h(\xi,\det\xi ) = h(\xi, \psi(\det \xi) )),
\end{equation*}
where $\psi$ is some smoothing of $\max(\epsilon_0,\cdot)$, so e.g. $\psi\in C^3(\R)$ with $\psi(t)=t$ for $t\geq \epsilon_0$, $\psi(t)=\epsilon_0/2$ for $t\leq \epsilon_0/2$ and $|\psi''|\leq C/\epsilon_0$. Then 
\begin{equation*}
\tilde e(\xi)=\tilde h(\xi,\det\xi)
\end{equation*}
fulfils the required properties.
\end{remark}

\begin{lemma}\label{lem:ineq-DE2}
We have for $C\equiv C(\eta_0,\eta_*,f)$ given by \eqref{eqn:time-reg-estimates} that
\begin{equation*}
\left \langle \fint_0^\tau\fint_0^h DE(\eta(t_{k-1}^\ell+s+\sigma)) -DE(\eta(t_k^\ell))\dx[s]\dx[\sigma],e_k^\ell-e_{k-1}^\ell \right \rangle  \leq C(h+\tau)\norm{e_k^\ell-e_{k-1}^\ell}.
\end{equation*}
\end{lemma}
\begin{proof}
We only provide here the estimate for $E_2$ the estimate for $E_1$ follows in the same way but needs much less regularity.
Recall the definition of $E_2$ in \eqref{eqn:Wk02-setting} and compute, given the regularity from \Cref{thm:higher-space-regularity}
\begin{multline*}
\langle \fint_0^\tau\fint_0^h DE_2(\eta(t_{k-1}^\ell+s+\sigma)) -DE_2(\eta(t_k^\ell))\dx[s]\dx[\sigma],e_k^\ell-e_{k-1}^\ell \rangle \\ = \fint_0^\tau \fint_0^h \int_Q \nabla^{k_0}(\eta(t_{k-1}^\ell+s+\sigma)-\eta(t_k^\ell)):\nabla^{k_0} (e_k^\ell-e_{k-1}^\ell) \dx\dx[s]\dx[\sigma]\\
=  \int_Q \fint_0^\tau \fint_0^h\nabla^{k_0}(\eta(t_{k-1}^\ell+s+\sigma)-\eta(t_k^\ell))\dx[s]\dx[\sigma]:\nabla^{k_0} (e_k^\ell-e_{k-1}^\ell) \dx\\
=  \int_Q \fint_0^\tau \fint_0^h(-\Delta)^{k_0}(\eta(t_{k-1}^\ell+s+\sigma)-\eta(t_k^\ell))\dx[s]\dx[\sigma]: (e_k^\ell-e_{k-1}^\ell) \dx
\end{multline*}
Now
\begin{multline*}
\norm{\fint_0^\tau \fint_0^h (-\Delta)^{k_0}(\eta(t_{k-1}^\ell+s+\sigma)-\eta(t_k^\ell))\dx[s]\dx[\sigma]}\leq \fint_0^\tau\fint_0^h (s+\sigma)\norm{\partial_t(-\Delta)^{k_0} \eta} \dx[s]\dx[\sigma]\\=
 \norm{\partial_t(-\Delta)^{k_0}\eta}_{L^\infty(((0,T);L^2(Q))}\frac{\tau+h}{2},
\end{multline*}
which implies the result by \Cref{thm:higher-space-regularity}.
\end{proof}

Finally we are in the position to show the convergence rate result.

\begin{theorem}[Convergence rate for elastic solid]\label{thm:convergence-rate-solid}Let the initial data satisfy  \begin{equation*}
\eta_0\in W^{2k_0,2}(Q),\quad \eta_*\in W^{ 2 k_0,2}(Q), \quad \partial_t f\in L^2((0,T);L^2(Q)).
\end{equation*}

There exists a $h_0>0$ and constants $C_1,C$ depending on $E(\eta_0)$, $\norm{\eta_*}_{L^2}$, $\norm{f}_{L^2((0,T);L^2(Q))}$ and $T$, such that for all $0<\tau\leq h\leq h_0$ (recall that we have $N\tau=h$ and $Mh=T$) the following convergence rate estimate holds
\begin{equation*}
\max_{k=1,\dots,N} \left( \frac12 \norm{\nabla^{k_0}e_k^\ell}^2 + \frac{1}{2N} \sum_{i=1}^k \norm{\frac{e_i^\ell-e_{i-1}^\ell}{\tau}}^2 \right) \leq (\tau^2+h^2)CT e^{C_1 h\ell}, \quad \ell=1,\dots,M.
\end{equation*}
\end{theorem}
\begin{proof}
Througout the proof $C$ is a constant, depending on on $E(\eta_0)$, $\norm{\eta_*}_{L^2}$, $\norm{f}_{L^2((0,T);L^2(Q))}$ and $T$.
From \eqref{eqn:convrate-aftertesting} we get, omitting the $\frac{1}{2h}\norm{\frac{e_k^\ell-e_{k-1}^\ell}{\tau}-\frac{e_k^{\ell-1}-e_{k-1}^{\ell-1}}{\tau}}^2 $ term and multiplying by $\tau$, 
\begin{multline*}
\frac{\tau}{2h}\norm{\frac{e_k^\ell-e_{k-1}^\ell}{\tau}}^2 + \frac{1}{2}\norm{\nabla^{k_0} e_k^\ell}^2 + \frac{1}{2} \norm{\nabla^{k_0} e_k^\ell-\nabla^{k_0} e_{k-1}^\ell}^2 +c\tau^2 \left\| \nabla\frac{e_k^\ell-e_{k-1}^\ell}{\tau} \right\|^2
\\ \leq \frac{\tau}{2h}\norm{\frac{e_k^{\ell-1}-e_{k-1}^{\ell-1}}{\tau}}^2 + \frac{1}{2}\norm{\nabla^{k_0} e_{k-1}^\ell}^2 -\langle DE_1(\eta_k^\ell)-DE_1(\eta(t_k^\ell)), e_k^\ell-e_{k-1}^\ell\rangle 
 \\ +\left\langle \fint_0^\tau\fint_0^h DE(\eta(t_{k-1}^\ell+s+\sigma)) -DE(\eta(t_k^\ell))\dx[s]\dx[\sigma],e_k^\ell-e_{k-1}^\ell\right\rangle \\ -c\tau^2 \left\langle\nabla \frac{\eta(t_k^\ell)-\eta(t_{k-1}^\ell)}{\tau}, \nabla\frac{e_k^\ell-e_{k-1}^\ell}{\tau} \right\rangle.
\end{multline*}
Using the inequalities from \Cref{lem:ineq-DE1} and \Cref{lem:ineq-DE2}
\begin{multline*}
\frac{\tau}{2h}\norm{\frac{e_k^\ell-e_{k-1}^\ell}{\tau}}^2 + \frac{1}{2}\norm{\nabla^{k_0} e_k^\ell}^2 + \frac{1}{2} \norm{\nabla^{k_0} e_k^\ell-\nabla^{k_0} e_{k-1}^\ell}^2+c\tau^2 \left\| \nabla\frac{e_k^\ell-e_{k-1}^\ell}{\tau} \right\|^2
\\ \leq \frac{\tau}{2h}\norm{\frac{e_k^{\ell-1}-e_{k-1}^{\ell-1}}{\tau}}^2\!\! + \frac{1}{2}\norm{\nabla^{k_0} e_{k-1}^\ell}^2\! +\tau C_1\norm{\nabla^{k_0} e_k^\ell}^2\!+  \tau C_2\norm{\frac{e_k^\ell-e_{k-1}^\ell}{\tau}}^2
\!+C(h+\tau)\norm{e_k^\ell-e_{k-1}^\ell}\\ -c\tau^2 \left\langle\nabla \frac{\eta(t_k^\ell)-\eta(t_{k-1}^\ell)}{\tau}, \nabla\frac{e_k^\ell-e_{k-1}^\ell}{\tau} \right\rangle
\end{multline*}
For the last term, thanks to the regularity of \Cref{thm:higher-space-regularity} $\Delta\frac{\eta(t_k^\ell)-\eta(t_{k-1}^\ell)}{\tau}= \frac1\tau\int_{t_{k-1}^\ell}^{t_k^\ell}\partial_t\Delta \eta \,dt $ get the estimate
\begin{multline*}
 -c\tau^2\left\langle \nabla\frac{\eta(t_k^\ell)-\eta(t_{k-1}^\ell)}{\tau},\nabla\frac{e_k^{\ell}-e_{k-1}^{\ell}}{\tau} \right\rangle=c\tau^2 \left\langle \Delta\frac{\eta(t_k^\ell)-\eta(t_{k-1}^\ell)}{\tau},\frac{e_k^{\ell}-e_{k-1}^{\ell}}{\tau} \right\rangle \\ \leq  c\tau^2 \|\partial_t\Delta\eta\|_{L^\infty((0,T);L^2(Q))} \left\| \frac{e_k^{\ell}-e_{k-1}^{\ell}}{\tau} \right\|.
 \end{multline*}
Finally use Young inequality and the estimate of \Cref{thm:higher-space-regularity}
$$ c\tau^2 \|\partial_t\Delta\eta\|_{L^\infty((0,T);L^2(Q))} \left\| \frac{e_k^{\ell}-e_{k-1}^{\ell}}{\tau} \right\| \leq C\tau^3+C\tau\left\| \frac{e_k^{\ell}-e_{k-1}^{\ell}}{\tau} \right\|^2.$$

By a further Young inequality 
\begin{equation*}
C(\tau+h)\norm{e_k^\ell-e_{k-1}^\ell}\leq \tau C (\tau^2+h^2) + \tau \frac12\norm{\frac{e_k^\ell-e_{k-1}^\ell}{\tau}}^2
\end{equation*}
we arrive at\begin{multline*}
\frac{\tau}{2h}\norm{\frac{e_k^\ell-e_{k-1}^\ell}{\tau}}^2 + \frac{1}{2}\norm{\nabla^{k_0} e_k^\ell}^2  +c\tau^2 \left\| \nabla\frac{e_k^\ell-e_{k-1}^\ell}{\tau} \right\|^2
\\ \leq \frac{\tau}{2h}\norm{\frac{e_k^{\ell-1}-e_{k-1}^{\ell-1}}{\tau}}^2\!\! + \frac{1}{2}\norm{\nabla^{k_0} e_{k-1}^\ell}^2\! +\tau C_1\norm{\nabla^{k_0} e_k^\ell}^2\!+  \tau C_2\norm{\frac{e_k^\ell-e_{k-1}^\ell}{\tau}}^2
\\ +\tau C (\tau^2+h^2) + \tau \frac12\norm{\frac{e_k^\ell-e_{k-1}^\ell}{\tau}}^2 +C\tau^3+C\tau\left\| \frac{e_k^{\ell}-e_{k-1}^{\ell}}{\tau} \right\|^2.
\end{multline*}

we get, denoting $a_k^\ell = \frac12 \norm{\nabla^{k_0}e_k^\ell}^2$, $b_k^\ell = \frac12\norm{\frac{e_k^\ell-e_{k-1}^\ell}{\tau}}^2$ that 
\begin{equation*}
a_k^\ell+\frac1N b_k^\ell \leq a_{k-1}^\ell+\frac1N b_k^{\ell-1} + \tau (C_2+1/2+C) b_k^\ell +\tau C(3\tau^2+h^2) .
\end{equation*}
Thus we can use Two-scale Gronwall inequality \Cref{twoscale-gronwall-classic} and obtain, since $a_0^0=0$, $b_0^0=0$, and $MN\tau =T$
\begin{equation*}
\max_{k=1,\dots,N} \left( \frac12 \norm{\nabla^{k_0}e_k^\ell}^2 + \frac{1}{2N} \sum_{i=1}^k \norm{\frac{e_i^\ell-e_{i-1}^\ell}{\tau}}^2 \right) \leq CT(\tau^2+h^2) e^{(C_2+1/2+C)h\ell}.
\end{equation*}
\end{proof}
%
%
%
%

\section{Numerical experiments}\label{sec:numerics}
We discuss some numerical experiments on the case of ODE in one dimension. This experiments do verify:

\begin{enumerate}
\item The {\em optimality of the rates} demonstrated in this paper to be qualitatively optimal already in the simplest possible set up. 
\item A characteristic danger in non-convex regimes, which is the possibility to ``land in the wrong well''; summarized in the necessity of the appearance of $\tau_0$ in the stability and convergence results.
\item The expected differences between the fully discrete, the time-delayed and the continuous solution. That is that the difference between the time-delayed solution and the time-discrete solution is of order $\tau$, while the difference between the time-delayed solution and the limit solution is of order $h$. This also indicates that choosing $\tau=h$ is commonly optimal with regard to convergence.
\end{enumerate} 
We consider a double-well potential, with minima at $1$ and $-1$
\begin{equation*}
E(x)=(x^2-1)^2.
\end{equation*}
We will compare the solution of
\begin{align*}
x''(t)+E'(x(t))&=0,\quad t\in (0,T)\\
x(0)&=x_0,\\ x'(0)&= x_*
\end{align*}
with that of our minimizing movements approximation $\overline x_{(\tau)}^{(h)}$. We will consider here the case $\tau=h$, which is according to our theory a good choice. Recall that our minimizing movements approximation is then defined as 
\begin{equation*}
x_k=\argmin_{x\in\R} \frac{\tau^2}{2}\abs{\frac{\frac{x-x_{k-1}}{\tau}-\frac{x_{k-1}-x_{k-2}}{\tau}}{\tau}}^2 + E(x),
\end{equation*}
where for $k=1$ we replace the fraction $\frac{x_0-x_{-1}}{\tau}$ by $x_*$, and the piecewise constant resp. piecewise affine interpolations $\overline x_{(\tau)}$ resp. $\hat x_{(\tau)}$ are defined as 
\begin{align*}
\overline x_{(\tau)}(t) &= x_k, \quad t\in ((k-1) \tau, k\tau],\\
\hat x_{(\tau)}(t) &= \frac{t-(k-1)\tau}{\tau} x_k + \frac{k\tau -t}{\tau }x_{k-1}, \quad  t\in [(k-1) \tau, k\tau].
\end{align*}

Further for comparison we include the solution of the \emph{time-delayed equation}
\begin{align*}
 \frac{x'(t)-x'(t-h)}{h} +E'(x(t))&=0,\\
 x(0)&=x_0,\\ x'(s)&=x_*,\ s\in(-h,0].
\end{align*}
\begin{figure}[!ht] 
\begin{center}    
	$h=\tau=1$:\\
    \includegraphics[width=0.6\paperwidth]{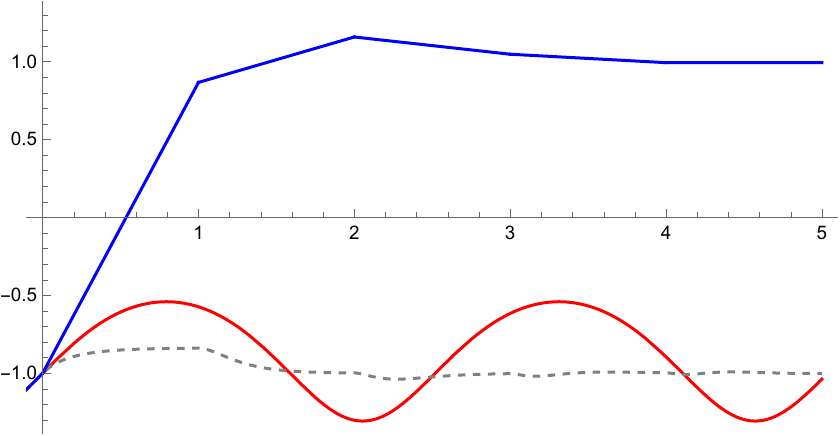}
    \end{center}
 \end{figure}
 \begin{figure}[!ht]
\begin{center}
    $h=\tau=1/5$:\\
    \includegraphics[width=0.6\paperwidth]{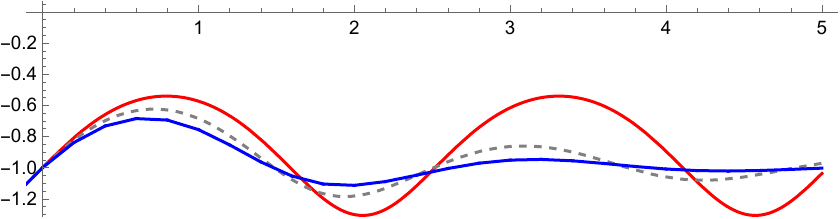}
    \end{center}
 \end{figure}
 \begin{figure}[!ht]
\begin{center}
    $h=\tau=10^{-2}$:\\    
    \includegraphics[width=0.6\paperwidth]{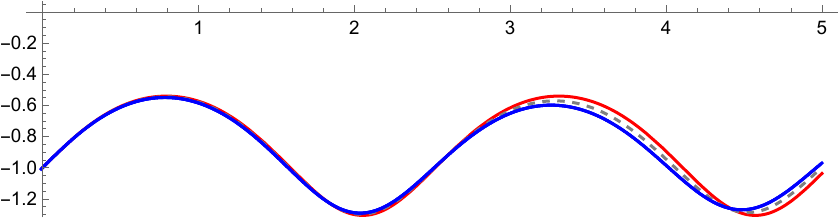}
    \end{center}
    \caption{Comparison of the our approximation $\hat \eta_{(\tau)}$ (blue) with the time-delayed solution (dashed) and the limit solution (red). Note that with too large parameter $h=\tau=1$ the solution overshoots and ends up in the other local minimum of $E$.}
\end{figure}
\begin{figure}[!ht]
\begin{tabular}{ | c |c|  } 
  \hline
  $\tau = h$ &  $\norm{x-\overline{x}_{\tau}}_{L^\infty((0,T))}$ \\ 
  \hline\hline
  $1$&  2.4589  \\ 
   \hline
 $10^{-1}$&  0.258051 \\ 
  \hline
 $10^{-2}$&  0.0821305 \\ 
  \hline
  $10^{-3}$&  0.00995943 \\ 
  \hline
\end{tabular}
\caption{Table showing the predicted error decay of the minimizing movements approximation, for different time steps.}
\end{figure}

\newpage
\subsection*{Acknowledgments}

\noindent
The authors acknowledge the support of the ERC-CZ grant LL2105, and the support of Charles University, project GA UK No.\@ 393421. S. S. further acknowledges the support of the VR-Grant 2022-03862 by the Swedish Research Council. A. \v C. further acknowledges the support of the Czech Science Foundation (GAČR) grant No. 23-04766S.

\bibliographystyle{alpha}
\bibliography{biblio}
\end{document}